\documentclass[11pt]{article}

\usepackage[T1]{fontenc}
\usepackage[utf8]{inputenc}
\usepackage[a4paper,margin=1in]{geometry}
\usepackage{amsmath,amssymb,amsfonts,amsthm,amsopn}
\usepackage{graphicx}
\usepackage{ifpdf}
\usepackage{xcolor}
\usepackage{natbib}
\usepackage{hyperref}
\usepackage{cleveref}
\usepackage{tikz}
\usetikzlibrary{arrows.meta,calc,decorations.pathreplacing}

\ifpdf
  \DeclareGraphicsExtensions{.pdf,.png,.jpg}
\else
  \DeclareGraphicsExtensions{.eps}
\fi

\numberwithin{equation}{section}
\newtheorem{theorem}{Theorem}[section]
\newtheorem{lemma}{Lemma}[section]

\title{Localized Pointwise A Posteriori Error Estimates for Nonconforming Finite Element Methods}
\author{Yongxing Guo\thanks{Email: \texttt{guoyongxing@zju.edu.cn}}
\qquad
Yuwen Li\thanks{Email: \texttt{liyuwen@zju.edu.cn}}\\[2pt]
\small School of Mathematical Sciences, Zhejiang University, 866 Yuhangtang Road,\\
\small Hangzhou, Zhejiang 310058, People's Republic of China}
\date{}

\ifpdf
\hypersetup{
  hidelinks,
  pdftitle={Localized Pointwise A Posteriori Error Estimates for Nonconforming Finite Element Methods},
  pdfauthor={Yongxing Guo and Yuwen Li}
}
\fi

\newcommand{\Th}{\mathcal{T}_h}
\newcommand{\Eh}{\mathcal{E}_h}
\newcommand{\Eho}{\mathcal{E}_h^\circ}
\newcommand{\Ehb}{\mathcal{E}_h^\partial}

\newcommand{\llbracket}{\mathopen{[\![}}
\newcommand{\rrbracket}{\mathclose{]\!]}}

\newcommand{\F}{\mathcal F}

\newcommand{\jump}[1]{\llbracket #1\rrbracket}

\newcommand{\dist}{\operatorname{dist}}
\newcommand{\supp}{\operatorname{supp}}
\newcommand{\diam}{\operatorname{diam}}

\newcommand{\dd}{\,\mathrm d}

\graphicspath{{figures/}}
\definecolor{farblue}{RGB}{94,155,203}
\definecolor{nearorange}{RGB}{230,153,73}

\definecolor{blueone}{RGB}{222,235,247}
\definecolor{bluetwo}{RGB}{198,219,239}
\definecolor{bluethree}{RGB}{174,202,226}
\definecolor{sourcepoint}{RGB}{190,45,45}

\begin{document}

\maketitle

\begin{abstract}
This paper establishes the first localized pointwise a posteriori error estimates for nonconforming finite element discretizations of the Poisson and biharmonic equations. For the Poisson problem, we derive localized estimates for the function-value and broken gradient errors of the Crouzeix--Raviart method. For the Morley discretization of the biharmonic equation, we derive a localized a posteriori estimate that controls the local Hessian error. This provides the first pointwise a posteriori error analysis for the biharmonic equation, for either conforming or nonconforming finite element methods. The key ingredient, absent from pointwise analysis of second order PDEs, is the design of two novel weight functions that facilitate sharp estimates of the \(L^1\) norms of derivatives of a regularized Green's function for the biharmonic operator.
\end{abstract}



\section{Introduction}
A posteriori error estimation provides computable information about the accuracy of finite element approximations and forms the basis of adaptive mesh refinement. Maximum-norm a posteriori bounds are particularly useful when pointwise error control is required. However, corner or edge singularities on polygonal domains may render the global maximum-norm error unbounded. Localized pointwise a posteriori estimates circumvent this limitation by controlling the error in a target region where the solution is locally regular. In practice, localized pointwise a posteriori error estimates are of particular interest because they can control finite element error within a local region of engineering interest. They also justify the local reliability of a posteriori error estimates up to higher-order perturbations, whereas most energy-norm error estimators in the literature are known to be reliable only globally.

The theory of pointwise error estimates for conforming discretizations of second-order elliptic problems is well developed; see \citep{Ciarlet1978,BrennerScott2008,SchatzWahlbin1978,DemlowLeykekhmanSchatzWahlbin2012,DieningRolfesSalgado2024} for pointwise a priori error estimates. \citet{Nochetto95} first established \(L^\infty\) a posteriori error estimates for conforming piecewise linear finite element approximations of Poisson's equation. For the same problem, \citet{HoffmannSchatzWahlbinWittum2001} proposed an elementwise asymptotically exact a posteriori estimator for the pointwise gradient error under a priori assumptions. Later, \citet{Demlow2007} derived a localized $W^{1,\infty}$ pointwise a posteriori error estimate for Poisson's equation, in which the local error is controlled by local $W^{1,\infty}$ error indicators and a global pollution term measured in a weaker norm.
The local a posteriori analysis in \citep{Demlow2007} was subsequently extended to the parabolic equation \citep{DemlowMakridakis2010} and the Stokes equations \citep{DemlowLarsson2013}. Local energy- and $L^2$-norm error estimators, together with convergence results for the corresponding adaptive methods, can be found in \citep{LiaoNochetto2003,XuZhou2004,Demlow2010}.

Global maximum-norm a posteriori error estimates for nonconforming finite element methods, such as the Crouzeix--Raviart (CR) method, have been established in \citep{DariDuranPadra2000}. Maximum-norm a priori error estimates for the CR method were derived in \citep{GastaldiNochetto1987}. Energy-norm residual-based and superconvergent recovery-type error estimators for the CR method can be found in \citep{CarstensenBartelsJansche2002,CaiHeZhang2017,Li2018SINUM,BankLi2019,HuMaMa2021}, while convergence analysis of the adaptive CR method was developed in \citep{becker2010convergent}. To the best of our knowledge, localized \(L^\infty\) and $W^{1,\infty}$ a posteriori error estimates for the CR method have not yet been established. 
In this paper, the desired estimates are obtained by combining the Galerkin projection property of the CR interpolation operator with the $L^1$ bounds for truncated regularized Green's functions.

Compared with second-order elliptic problems, the theory of pointwise error estimates for biharmonic problems is much less developed. Nonconforming methods are widely used for fourth-order equations because implementing $C^1$ finite elements is expensive and difficult. Relevant maximum-norm-type a priori error estimates have been derived for the Morley method \citep{Rannacher1979,Wang1993}, mixed methods \citep{Rannacher1979}, conforming methods \citep{LiAnLi2007}, and the $C^0$ interior penalty method \citep{Leykekhman2021}. To the best of our knowledge, no global or localized pointwise a posteriori error analysis is available for the biharmonic equation. In contrast, energy-norm a posteriori error estimates for finite element discretizations of biharmonic equations, including the Morley method, are well established \citep{BeiraoNiiranenStenberg2007,HuShi2009,HuMa2016,Li2021JSCb}. 

To fill this gap, we establish the first pointwise a posteriori error estimate that controls the localized Hessian error of the Morley method. Due to insufficient elliptic regularity, a global $W^{2,\infty}$ a posteriori error estimate is impossible even on convex domains. The analysis of the biharmonic problem is more difficult than that of second-order elliptic equations because the biharmonic operator is of higher order and the Morley finite element is more strongly nonconforming. The key ingredient, absent from the localized pointwise analysis of second order PDEs, is the design of two novel weight functions $\Theta_T$ and $\Phi_T$ that facilitate sharp estimates of the \(L^1\) norms of derivatives of a regularized Green's function for the biharmonic equation.
As far as we know, the estimates derived here for the regularized biharmonic Green's functions and the two weight functions are new; see Sections \ref{sect:Green} and \ref{sect:Morley} for details.

\subsection{Main results}
Let $\Omega \subset\mathbb{R}^2$ be a polygonal domain, and let \((\cdot,\cdot)\) denote the \(L^2(\Omega)\) inner product. Let $D^kv$ denote the $k$th-order tensor comprising all $k$th-order derivatives of $v$; for example, $Dv=D^1v$ is the gradient and $D^2v$ is the Hessian. Given $f\in L^\infty(\Omega)$, we consider the weak form of Poisson's equation with homogeneous Dirichlet boundary conditions: find $u\in H_0^1(\Omega)$ such that
	\begin{equation*}
		a(u,v)=(D u, D v)=(f,v), \quad \forall v\in H_0^1(\Omega).
	\end{equation*}
We also consider the weak formulation of the biharmonic equation with clamped boundary conditions: find $U\in H_0^2(\Omega)$ such that
	\begin{equation*}
		b(U,v)=(D^2U, D^2v) =(f,v), \quad \forall v\in H_0^2(\Omega).
	\end{equation*}

Let \(\mathcal T_h\) be a conforming triangulation of \(\Omega\). Denote by \(\mathcal V_h\) and \(\mathcal E_h\) the sets of all mesh vertices and edges, respectively. For a set $U$, let $\mathcal{E}_h(U)$ denote the set of edges contained in $U$. Let \(\mathcal V_h^\circ\), \(\mathcal V_h^\partial\), \(\mathcal E_h^\circ\), and \(\mathcal E_h^\partial\) denote the subsets of interior vertices, boundary vertices, interior edges, and boundary edges, respectively. 
For each edge \(F \in \mathcal E_h\), let \(n=n_F\) and \(t=t_F\) denote the unit normal and unit tangent vectors along $F$, respectively. For an interior edge \(F=T^+\cap T^-\), \(\llbracket  \cdot \rrbracket_F\) denotes the jump across \(F\). When no confusion arises, we omit the subscript \(F\).
We assume that the mesh is shape-regular, namely, there exists a constant \(\gamma_{\rm sh}>0\), independent of \(h\), such that 
	\begin{equation*}
		\frac{h_T}{\rho_T}\le \gamma_{\rm sh}, 
		\qquad \forall T\in\mathcal T_h .
	\end{equation*}

Consider the Crouzeix--Raviart finite element space
	\begin{equation*}
		\begin{split}
			V_h:=\biggl\{&v_h\in L^2(\Omega):\ v_h|_T\in\mathbb{P}_1(T)\, \forall T\in\mathcal{T}_h,\\
			&\int_F\jump{v_h}\,\dd s=0\ \forall F\in\mathcal{E}_h^\circ,\ \int_Fv_h\,\dd s=0\ \forall F\in\Ehb\biggr\}
		\end{split}
	\end{equation*}
and the Morley finite element space
	\begin{equation*}
		\begin{split}
			W_h:=\biggl\{&v_h\in L^2(\Omega):\ v_h|_T\in\mathbb P_2(T)\ \forall T\in\Th,\ v_h\text{ is single-valued at every }a\in\mathcal V_h^\circ,\\
			&\int_F\jump{\partial_n v_h}\,\dd s=0\ \forall F\in\Eho, v_h(a)=0\ \forall a\in\mathcal V_h^\partial,\ \int_F\partial_n v_h\,\dd s=0\ \forall F\in\Ehb\biggr\}.
		\end{split}
	\end{equation*}
Here $\mathbb{P}_r(T)$ is the space of polynomials of degree at most $r$ on $T$, and $\partial_n$ denotes the directional derivative along $n=n_F$.

Let \(D_h^k\) denote the broken \(k\)th-order derivative tensor on \(\mathcal{T}_h\). The CR method for the Poisson equation seeks \(u_h\in V_h\) such that
\begin{equation}\label{eq:CR}
a_h(u_h,v_h):=(D_hu_h,D_hv_h)=(f,v_h),\qquad \forall v_h\in V_h.
\end{equation}
The Morley method seeks \(U_h\in W_h\) satisfying
	\begin{equation}\label{eq:Morley}
		b_h(U_h,v_h):=(D_h^2U_h,D_h^2v_h)=(f,v_h),\qquad \forall v_h\in W_h.
	\end{equation}

We establish localized pointwise a posteriori error estimates for \eqref{eq:CR} and \eqref{eq:Morley}. 
Let \(\omega\subset\Omega\) be an open target region with smooth boundary, and let \(d>0\) be a localization radius. Define \(\omega_d=\{x\in\Omega:\operatorname{dist}(x,\omega)<d\}\). We impose one of the following conditions: either \(\operatorname{dist}(\omega_d,\partial\Omega)>d\), or \(\omega_d\) meets \(\partial\Omega\) only along a flat boundary portion and remains separated from reentrant corners and edges on the scale \(d\). Let \(h_F=|F|\), \(h_T=\operatorname{diam}T\), and \(\rho_T\) denote the inradius of an element 
	\(T\in\mathcal{T}_h\).
For the local mesh sizes near \(\omega\), set
    \[
    \underline h_{\omega}:=\min\{h_T:T\cap \omega\neq\emptyset\},
    \qquad
    \overline h_{\omega}:=\max\{h_T:T\cap \omega\neq\emptyset\}.
    \]
We assume the local resolution condition $\overline h_{\omega}\le d$. The local broken maximum norm is
\[
\|v\|_{L^\infty(\omega;\mathcal T_h)}
:=\max_{T\in\mathcal T_h}\|v\|_{L^\infty(\omega\cap T)}.
\]

The local error indicator for the CR method is defined by
    \begin{subequations}
	\begin{equation*}
		\eta_0(T):= h_T^2\|f\|_{L^\infty(T)}+ \max_{F\subset\partial T}\|\jump{u_h}\|_{L^\infty(F)},
	\end{equation*}
	\begin{equation*}
		\eta_1(T):=h_T\|f\|_{L^\infty(T)}
		+h_T^{-1} \max_{F\subset\partial T}\|\jump{u_h}\|_{L^\infty(F)}.
	\end{equation*}
    \end{subequations}
For \(T\in\mathcal T_h\), the corresponding Morley estimator is defined by
\begin{align*}
\zeta(T):&=h_T^2\|f\|_{L^\infty(T)}+\max_{F\subset\partial T} h_F^{-1} \|\llbracket\partial_n U_h\rrbracket \|_{L^\infty(F)}+\max_{F\subset\partial T} h_F^{-2} \|\llbracket U_h\rrbracket \|_{L^\infty(F)}.
\end{align*}

We assume that the mesh is shape-regular, namely, there exists a constant \(\gamma_{\rm sh}>0\), independent of \(h\), such that $\max_{T\in\mathcal{T}_h}(h_T/\rho_T)\le \gamma_{\rm sh}$ for all $T\in\mathcal T_h$. 
Throughout this paper, \(A_1\lesssim A_2\) means that \(A_1\le C A_2\) for some generic constant \(C>0\) independent of the mesh size and the quantities under consideration. We write \(A_1\eqsim A_2\) when \(A_1\lesssim A_2\) and \(A_2\lesssim A_1\). Let $L_{\rho,d}=1+\log(d/\rho)$. 

For an integer $k\ge0$, $1\le p\le\infty$, and the conjugate exponent $q$, we use the negative norm
\begin{equation*}
\|w\|_{W^{-k,p}(G)}
:=\sup_{0\ne v\in C_0^\infty(G)}
\frac{|\langle w,v\rangle|}{\|v\|_{W^{k,q}(G)}},
\qquad \frac1p+\frac1q=1,
\end{equation*}
with the usual endpoint conventions $1/\infty=0$.

Our main theoretical results are stated in the next two theorems.
\begin{theorem}\label{thm:main_CR}
Suppose that \(u\in C^{2,\alpha}(\omega_{d})\) for some \(0<\alpha\le1\). 
For every integer \(k\ge0\) and every \(1\le p\le\infty\), we have
		\begin{equation*}
        \begin{aligned}
            \|u-u_h\|_{L^\infty(\omega;\mathcal T_h)}
			&\lesssim L_{\underline{h}_\omega,d}\max_{T\cap \omega_{d}\ne\emptyset}\eta_0(T)+d^{-k-2/p}\|u-u_h\|_{W^{-k,p}(\omega_{d})}\\
            &+\max_{T\cap \omega_{d}\ne\emptyset}h_T^{2+\alpha}|u|_{C^{2,\alpha}(\omega_{d})},\\
			\|D_{h}(u-u_h)\|_{L^\infty(\omega;\mathcal T_h)}
			&\lesssim L_{\underline{h}_\omega,d}\max_{T\cap \omega_{d}\ne\emptyset}\eta_1(T)
			+d^{-1-k-2/p}\|u-u_h\|_{W^{-k,p}(\omega_{d})}\\
            &+\max_{T\cap \omega_{d}\ne\emptyset}h_T^{1+\alpha}|  u|_{C^{2,\alpha}(\omega_{d})}.
			\end{aligned}
		\end{equation*}
	\end{theorem}

\begin{theorem}\label{thm:main_Morley}
Suppose that \(U\in C^{3,\alpha}(\omega_{d})\) for some \(0<\alpha\le1\). Then, for every integer \(k\ge0\) and \(1\le p\le\infty\), we have
		\begin{equation*}
			\begin{aligned}
                \|D_h^2(U-U_h)\|_{L^\infty(\omega;\mathcal T_h)}&\lesssim  L_{\underline h_{\omega},d} \max_{T\cap \omega_{d}\ne\emptyset} \zeta(T)+d^{-2-k-2/p}\|U-U_h\|_{W^{-k,p}(\omega_{d})}\\
                &+\max_{T\cap \omega_{d}\ne\emptyset} h_T^{1+\alpha}|U|_{C^{3,\alpha}(\omega_{d})}.
			\end{aligned}
		\end{equation*}
\end{theorem}

A posteriori error estimates in Theorems \ref{thm:main_CR} and \ref{thm:main_Morley} are local because
the errors $\|u-u_h\|_{W^{-k,p}(\omega_{d})}$ and $\|U-U_h\|_{W^{-k,p}(\omega_{d})}$ in negative norms, together with the regularization terms involving H\"older seminorms, are of higher order. We also present a posteriori lower error bounds in Sections \ref{sect:CR} and \ref{sect:Morley}.
All results established in this paper extend naturally to three dimensions, with only minor modifications to the arguments.

The rest of the paper is organized as follows. Section~\ref{sect:Green} introduces the necessary estimates for regularized Green's functions. Section~\ref{sect:CR} establishes localized a posteriori error estimates for the CR method. Section~\ref{sect:Morley} develops localized pointwise a posteriori error estimates for the Morley method. Numerical experiments are presented in Section~\ref{sect:NE}. Section \ref{sect:conclusion} contains concluding remarks.

\section{Regularized Green's functions}\label{sect:Green}
To represent pointwise values, we need to analyze regularized Green's functions. Given a point $x_0\in\overline{\omega}$, choose \(T_0\in\mathcal{T}_h\) such that
\(x_0\in\overline T_0\) and $T_0\cap\omega\ne\emptyset$, and let
\(h_{T_0}\lesssim\rho\le \kappa h_{T_0}\), where \(\kappa>0\) is sufficiently small. Taking $\kappa\le1$, the local resolution condition gives
\begin{equation*}
\rho\le \kappa h_{T_0}\le h_{T_0}\le\overline h_\omega\le d.
\end{equation*}

There exists a regularized Dirac function \(\delta_{x_0,\rho}\in C^\infty_0(T_0)\), supported in a disk of radius \(\rho\), such that
\begin{subequations}
\begin{align}
\int_{T_0}\delta_{x_0,\rho}P\,\dd x=P(x_0),\qquad \forall P\in\mathbb P_s(T_0),\label{eq:delta_moments}\\
\|D^j\delta_{x_0,\rho}\|_{L^p(T_0)}\lesssim \rho^{-j-2(1-1/p)},\qquad j\ge0,\quad 1\le p\le\infty,\label{eq:delta_Lp}
\end{align}
\end{subequations}
see \citep{GiraultNochettoScott2005,Demlow2007} for a detailed construction.
In the CR case, we take \(s=2\), whereas in the Morley case, we take \(s=3\).

We consider regularized Green's functions on the disk
\(B:=B(x_0,d/2)\subset\omega_d\). Thus, $B$ is centered at $x_0$ and
$\diam(B)=d$. The support of $\delta_{x_0,\rho}$ is contained in $B$, but it is not necessarily centered at $x_0$.
Figure \ref{fig:localization-geometry} illustrates these concepts in localized error estimates. 

\begin{figure}[htbp]
\centering
\begin{tikzpicture}[
    scale=0.8,
    transform shape,
    x=1cm,y=1cm,
    every node/.style={font=\small},
    >={Latex[length=2mm]}
]
% The curved nonconvex physical domain Omega
\path[
    fill=gray!4,
    draw=black!75,
    line width=0.9pt
]
    (0.65,1.10)
    .. controls (0.05,2.25) and (0.15,5.60) .. (1.30,6.85)
    .. controls (2.80,8.05) and (6.70,7.75) .. (9.15,6.95)
    .. controls (10.40,6.55) and (10.65,5.45) .. (9.45,5.05)
    .. controls (8.55,4.75) and (7.95,5.55) .. (7.30,5.00)
    .. controls (6.75,4.50) and (8.10,3.85) .. (9.30,3.40)
    .. controls (10.70,2.85) and (10.25,1.15) .. (8.65,0.70)
    .. controls (6.45,0.05) and (2.20,0.10) .. (0.65,1.10)
    -- cycle;
\node[font=\Large] at (8.55,6.2) {$\Omega$};

% Centers of omega and the distinguished point
\coordinate (xomega) at (4.35,4.00);
\coordinate (x0) at (5.65,3.65);

% The d-neighborhood omega_d
% The new value of d is 2.35-1.65=0.70
\path[
    fill=blue!9,
    draw=blue!65!black,
    line width=0.9pt
]
    (xomega) circle[radius=2.35];
\node[,blue!65!black] at (5.75,5.45) {$\omega_d$};

% The target region omega
\path[
    fill=orange!24,
    draw=orange!80!black,
    line width=0.9pt
]
    (xomega) circle[radius=1.65];
\node[font=\Large,orange!80!black] at (4.35,4.0) {$\omega$};

% Distance between the boundaries of omega and omega_d
\draw[
    blue!65!black,
    <->,
    line width=0.8pt
]
    ($(xomega)+(205:1.65)$) --
    ($(xomega)+(205:2.35)$)
    node[midway,below left,xshift=9pt,
        yshift=1pt] {$d$};

% The local Green's function domain B
\path[
    fill=green!18,
    fill opacity=0.72,
    draw=green!45!black,
    line width=1pt
]
    (x0) circle[radius=0.825];
\node[green!35!black] at (5.4,4.03) {$B$};

% The source element T_0
\path[
    fill=gray!35,
    fill opacity=0.65,
    draw=black!70,
    line width=0.9pt
]
    ($(x0)+(-0.425,-0.275)$) --
    ($(x0)+( 0.325,-0.225)$) --
    ($(x0)+( 0.165, 0.625)$) -- cycle;
\node at (6.09,3.72) {$T_0$};

% Redraw the boundary of B
\draw[green!45!black,line width=1pt]
    (x0) circle[radius=0.825];
% Support of the regularized Dirac function
\path[
    fill=red!48,
    draw=red!75!black,
    line width=0.8pt
]
    ($(x0)+(0.0,-0.0)$) circle[radius=0.14];

\draw[->,red!75!black,line width=0.8pt]
    ($(x0)+(1.58,-0.45)$) -- ($(x0)+(0.18,-0.10)$);

\node[font=\large,anchor=west,red!75!black]
    at ($(x0)+(1.58,-0.45)$) {$\supp \delta_{x_0,\rho}$};

% The distinguished point x_0
\fill ($(x0)+(0.03,-0.01)$) circle[radius=1.8pt];
\node[anchor=north east]
    at ($(x0)+(-0.1,0.25)$) {$x_0$};
\end{tikzpicture}

\caption{Schematic diagram of localized pointwise estimates.}
\label{fig:localization-geometry}
\end{figure}
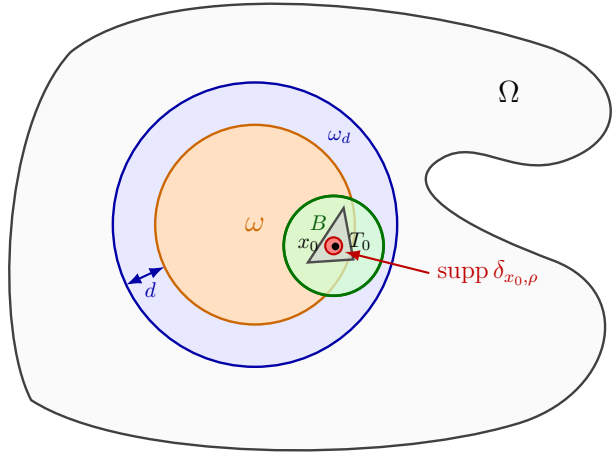

For the Poisson problem, let \(m\in\{0,1\}\) and let \(\gamma\) be a multi-index with \(|\gamma|=m\). Define \(g_m\in H^1_0(B)\) by
\begin{equation}\label{eq:local-poisson-green-strong}
		\begin{split}
		-\Delta g_m&=(-1)^mD^{\gamma}\delta_{x_0,\rho}\quad\text{in }B,\\
		g_m&=0  \qquad\qquad\qquad\:\:\: \text{on }\partial B.
		\end{split}
\end{equation}
For the biharmonic problem, we set \(|\gamma|=2\) and define \(\mathfrak{g}\in H^2_0(B)\) by
\begin{equation}\label{eq:local-biharmonic-green}
		\begin{split}
	\Delta^2\mathfrak{g}&=D^{\gamma}\delta_{x_0,\rho}\quad\text{in }B,\\
			\mathfrak{g}&=\partial_n \mathfrak{g}=0\quad\text{on }\partial B.
		\end{split}
	\end{equation}
The Green's function estimates for the Poisson equation are taken from \citep{krasovskii1969properties}, while those for the biharmonic equation are taken from \citep{dall2004estimates}.
\begin{lemma}\label{lem:standard-green-kernel-bound}
	Let \(G(x,y)\) be the Dirichlet Green's function on \(B\), namely
\begin{equation*}
\begin{aligned}
-\Delta_x G(x,y)&=\delta_y &&\text{in } B,\\
G(x,y)&=0 &&\text{on }\partial B,
\end{aligned}
\end{equation*}
where $\delta_y$ is the Dirac delta at $y$.
Then for all \(x,y\in B\) with \(x\neq y\), we have
    \begin{align*}
        |G(x,y)| &\lesssim1+\left|\log\dfrac{|x-y|}{d}\right|,\\
		|D_x^\beta G(x,y)|&\lesssim |x-y|^{-|\beta|},\quad |\beta|\ge1.
    \end{align*}
\end{lemma}

\begin{lemma}\label{lem:biharmonicGreen_regularized}
	Let  \(\mathcal G(x,y)\) be the Green's function of the clamped biharmonic problem on \(B\), namely,
\[
	\begin{aligned}
		\Delta_x^2\mathcal G(x,y)&=\delta_y \qquad \text{in } B,\\
		\mathcal G(x,y)&=\partial_n\mathcal G(x,y)=0 \qquad \text{on }\partial B.
	\end{aligned}
\]
Then for all \(x,y\in B\) with \(x\neq y\), we have
	\begin{equation*}
		\left|D_x^\beta D_y^\gamma\mathcal G(x,y)\right|
		\lesssim
		\begin{cases}
			d^{2-|\beta|-|\gamma|}, &  |\beta|+|\gamma|\le 1,\\[1mm]
			1+\left|\log\dfrac{|x-y|}{d}\right|, & |\beta|+|\gamma|=2,\\[2mm]
			|x-y|^{2-|\beta|-|\gamma|}, & |\beta|+|\gamma|\ge 3.
		\end{cases}
	\end{equation*}
\end{lemma}

We next derive pointwise upper bounds for the regularized Green's functions.
We give the proof only for the more difficult biharmonic problem; the proof for the Poisson problem is analogous.
\begin{lemma}
	\label{lem:pointwise-green-estimates-displacement}
For every $x\in B$, we have
    \begin{equation*}
        |D^j \mathfrak{g}(x)|\lesssim
        \begin{cases}
            \displaystyle 1+\left|\log\frac{|x-x_0|+\rho}{d}\right|,
            & j=0,\\ \bigl(|x-x_0|+\rho\bigr)^{-j},
            & j\ge 1.
        \end{cases}
    \end{equation*}
\end{lemma}

\begin{proof}
By the definition \eqref{eq:local-biharmonic-green}, integration by parts with respect to \(y\), and differentiation with respect to \(x\), we obtain
	\[
	D^j\mathfrak{g}(x)=\int_BD_x^jD_y^\gamma\mathcal G(x,y)\delta_{x_0,\rho}(y)\,\dd y,\qquad j\ge 1.
	\]

For the outer region \(\{x\in B: |x-x_0|\ge 2\rho\}\), we have \(|x-y|\eqsim |x-x_0|\) for \(y\in\operatorname{supp}\delta_{x_0,\rho}\), and consequently \(|x-y|\eqsim |x-x_0|+\rho\). The pointwise estimates in Lemma \ref{lem:biharmonicGreen_regularized} and \eqref{eq:delta_Lp} therefore yield
	\[
    |D^j \mathfrak{g}(x)|\le\int_{T_0}|D_x^jD_y^\gamma\mathcal G(x,y)|\,|\delta_{x_0,\rho}(y)|\,\dd y\lesssim
    \left\{\begin{aligned}
        & 1+\left|\log\frac{|x-x_0|+\rho}{d}\right|,\quad j=0,\\
        &\bigl(|x-x_0|+\rho\bigr)^{-j},\quad j \ge 1.
    \end{aligned}\right.
	\]
	
For the inner region \(\{x\in B: |x-x_0|<2\rho\}\), we introduce the scaled variable $\widehat x=(x-x_0)/\rho$ and transformation $\widehat{f}(\widehat x)=f(x_0+\rho\widehat x)$.
Then \(\mathfrak{\widehat g} \) satisfies 
\[
\Delta_{\widehat x}^2\mathfrak{\widehat g} =\rho^2 D_{\widehat x}^{\gamma}\widehat\delta_{x_0,\rho}\quad\text{in }\widehat{B}=B(0,d/\rho).
\]

Let $\widehat{S}=\operatorname{supp}\widehat\delta_{x_0,\rho}$, and let $\mathcal{G}_{\widehat B}$ be the biharmonic Green's function on $\widehat{B}$.
By \eqref{eq:delta_Lp},
\begin{equation*}
\|D_{\widehat x}^l \widehat\delta_{x_0,\rho}\|_{L^\infty}
\lesssim \rho^{l} \|D_x^l \delta_{x_0,\rho}\|_{L^\infty}
\lesssim \rho^{-2}.
\end{equation*}
Combining this bound with the symmetry of the Green's function and integrating by parts yields
\begin{align*}
|D^j_{\widehat x}\mathfrak{\widehat g}(\widehat x)|
&=\left|D^j_{\widehat x}\int_{\widehat S}\mathcal G_{\widehat B}(\widehat x,\widehat y) \, \rho^2 D^\gamma_{\widehat y} \widehat\delta_{x_0,\rho}(\widehat y)\,\dd \widehat y\right|
\\
&=\rho^2 \left|\int_{\widehat S} D_{\widehat y} D^{\gamma}_{\widehat y}\mathcal G_{\widehat B}(\widehat y,\widehat x) \, D^{j-1}_{\widehat y} \widehat\delta_{x_0,\rho}(\widehat y)\,\dd \widehat y\right|\\
&\lesssim \rho^2 \|D_{\widehat y}^{j-1} \widehat\delta_{x_0,\rho}\|_{L^\infty} \int_{\widehat S}|D_{\widehat y} D^\gamma_{\widehat y} \mathcal G_{\widehat B}(\widehat x,\widehat y)| \,\dd \widehat y 
\lesssim \int_{\widehat S} |\widehat x-\widehat y|^{-1} \,\dd \widehat y 
\lesssim 1
\end{align*}
for $j \ge 1$.
Similarly, for $j=0$ we obtain
\begin{align*}
|\mathfrak{\widehat g}(\widehat x)|
\lesssim \|\rho^2  \widehat\delta_{x_0,\rho}\|_{L^\infty} \int_{\widehat S}|D^2_{\widehat x} \mathcal G_{\widehat B}(\widehat x,\widehat y)| \,\dd \widehat y
\lesssim1+\log\frac d\rho.
\end{align*}
Rescaling and using \(|x-x_0|+\rho\eqsim\rho\) in the inner region, we obtain 
\begin{align*}
	|\mathfrak{g}(x)|&\lesssim1+\log\frac d\rho\eqsim 1+\left|\log\frac{|x-x_0|+\rho}{d}\right|,\\
    |D^j\mathfrak{g}(x)|&\lesssim\rho^{-j}\eqsim \bigl(|x-x_0|+\rho\bigr)^{-j},\quad j\ge 1.    
\end{align*}
The proof is complete. 
\end{proof}

\begin{lemma}\label{lem:pointwise-green-estimates}
For every $x\in B$, we have
\begin{equation*}
    \begin{aligned}
        |D^j g_1(x)| &\lesssim (|x-x_0|+\rho)^{-1-j},\quad j\ge0,\\
        |D^j g_0(x)| &\lesssim  
        \left\{\begin{aligned}
        &1+\left|\log\frac{|x-x_0|+\rho}{d}\right|,  \qquad j=0,\\
        &(|x-x_0|+\rho)^{-j},\qquad \qquad \, j\ge1.
        \end{aligned}\right.
    \end{aligned}
	\end{equation*}
\end{lemma}

\subsection{Truncated Regularized Green's Function}
\label{subsec:localization}
We introduce a cut-off function \(\xi\in C_0^\infty(B)\) with $\xi|_{B^c}\equiv0$ satisfying
    \begin{subequations}
    \begin{align}
    &\xi\equiv 1 \quad\text{on }\operatorname{supp}\delta_{x_0,\rho},\qquad\xi=0 \quad\text{on } \partial B, \label{eq:cutoff-one}\\
	&\|D^j\xi\|_{L^\infty(B)} \lesssim  d^{-j},\qquad\qquad\qquad\qquad\qquad\qquad j\ge 0 ,
    \label{eq:cutoff-bounds}\\
    &\operatorname{supp}(D^j\xi) \subset
	\left\{ x\in B: c_1d\le |x-x_0|\le c_2d \right\}, \,\,\,\, \quad j\geq 1,
    \end{align}
    \end{subequations}
where $c_1, c_2\in(0,1)$ are absolute constants. 

We now present pointwise bounds for the truncated Green's functions  
\begin{equation*}
\phi_m:=\xi g_m, \qquad\psi:=\xi \mathfrak{g}
\end{equation*}
in the next three lemmas. 
\begin{lemma}
	\label{lem:D2phi0}
    For the truncated Green's function \(\phi_0\), we have \begin{equation*}\|D^2\phi_0\|_{L^1(B)}\lesssim 1+\log\frac d\rho.
	\end{equation*}
\end{lemma}

\begin{proof}
By Lemma~\ref{lem:pointwise-green-estimates}, we have the pointwise bound
\[|D^2g_0(x)| \lesssim \bigl(|x-x_0|+\rho\bigr)^{-2}.\]
We decompose \(B\) into the inner region \(A_{-1}=\overline {B(x_0,\rho)}\) and the dyadic annuli
\begin{equation}\label{dyadic-annuli}
	    A_i = \{x\in B:2^i\rho<|x-x_0|\le 2^{i+1}\rho\},
	\qquad i=0,\ldots,M,
	\end{equation}
	where \(2^M\rho\eqsim d\). Thus, \(M\lesssim\log(d/\rho)\).

In the inner region \(A_{-1}\), we have \(|x-x_0|\leq \rho\) and thus 
	\[
	\|D^2g_0\|_{L^1(B(x_0,\rho))}\lesssim \rho^{-2}|B(x_0,\rho)|
	\lesssim  1.
	\]
On each annulus \(A_i\), we have \(|x-x_0|\eqsim 2^i\rho\), and hence
\[
\|D^2g_0\|_{L^1(A_i)} \lesssim 	(2^i\rho)^{-2}|A_i|\eqsim1.
\]
The number of annuli is bounded by \(O(1+\log(d/\rho))\). Hence
\begin{equation}\label{eq:D2g-L1-estimate}
		\|D^2g_0\|_{L^1(B)} \lesssim1+\log\frac d\rho.
	\end{equation}
It remains to pass from \(g_0\) to \(\phi_0=\xi g_0\). The product rule gives
	\[
	D^2\phi_0	=\xi D^2g_0+2 D \xi\otimes D  g_0+ g_0 D^2\xi .
	\]
The first term is bounded by \eqref{eq:D2g-L1-estimate}. The remaining terms are supported in the cut-off region, where \(|x-x_0|\eqsim d\). Using Lemma \ref{lem:pointwise-green-estimates}, \eqref{eq:cutoff-bounds}, and the fact that this region has measure \(O(d^2)\), we obtain
\[
	\| D \xi\otimes D  g_0\|_{L^1(B)}\lesssim d^{-1}d^{-1}d^2 \lesssim 1,
	\qquad
	\|g_0D^2\xi\|_{L^1(B)} \lesssim d^{-2}d^2\lesssim1 .
	\]
The proof is complete.
\end{proof}

\begin{lemma}
	\label{lem:D2phi-gradient-weighted-L1-estimate}
For the truncated Green's function \(\phi_1\), we have the weighted estimate
	\begin{equation*}
		\sum_{T\cap \omega_d\ne\emptyset}
		h_T\|D^2\phi_1\|_{L^1(T)} \lesssim 1+\log\frac d\rho.
	\end{equation*}
\end{lemma}

\begin{proof}
By Lemma~\ref{lem:pointwise-green-estimates} with \(j=2\), we have
	\[
	|D^2g_1(x)|\lesssim \bigl(|x-x_0|+\rho\bigr)^{-3}.
	\]
We first consider the elements intersecting \(B(x_0,\rho)\). Since $\rho\eqsim h_{T_0}$, we have \(h_T\lesssim \rho\) and thus
\[
\sum_{T\cap B(x_0,\rho)\ne\emptyset}h_T \|D^2g_1\|_{L^1(T)}
	\lesssim \rho\cdot\rho^{-3}|B(x_0,\rho)|\lesssim 1.
\]
For the remaining elements away from $B(x_0,\rho)$, it has been shown in \citep{Demlow2007} that
\[
\sum_{T\cap\omega_d\neq\emptyset,\, T\cap B(x_0,\rho)=\emptyset}h_T\|D^2g_1\|_{L^1(T)}\lesssim 1+\log\frac{d}{\rho}.
\]
Combining the previous two estimates, we obtain
\begin{equation}\label{eq:weighted-D2gi-estimate}
		\sum_{T\cap \omega_d\ne\emptyset}h_T\|D^2g_1\|_{L^1(T)}
		\lesssim1+\log\frac d\rho.
	\end{equation}

It remains to pass from \(g_1\) to \(\phi_1=\xi g_1\). We have
	\[
	D^2\phi_1=\xi D^2g_1+2 D \xi\otimes D  g_1+g_1D^2\xi.
	\]
The first term is bounded by \eqref{eq:weighted-D2gi-estimate}. Using the estimates in Lemma \ref{lem:pointwise-green-estimates} and \eqref{eq:cutoff-bounds}, we obtain
\[
\begin{aligned}
 \sum_{T\cap \omega_d\ne\emptyset} h_T\|2 D \xi\otimes D  g_1\|_{L^1(T)}
&\lesssim d d^{-1} d^{-2}d^2\lesssim 1,\\
\sum_{T\cap \omega_d\ne\emptyset} h_T\|g_1D^2\xi\|_{L^1(T)}
&\lesssim d  d^{-2}d^{-1}  d^2\lesssim 1.
\end{aligned}
\]
Combining these three estimates proves Lemma \ref{lem:D2phi-gradient-weighted-L1-estimate}.
\end{proof}
%%%%%%%%%%%%%%%%%%%%%%

\begin{lemma}\label{lem:Djpsi}
For the truncated Green's function \(\psi\), we have
	\begin{equation*}
		|D^j\psi (x)|
		\lesssim \bigl(|x-x_0|+\rho\bigr)^{-j},\quad j\geq1.
	\end{equation*}
\end{lemma}

\begin{proof} 
We consider the source region, the cut-off transition region, and the region where the cut-off function vanishes separately. In the source region, \(\xi\equiv 1\), so the asserted estimate follows from Lemma~\ref{lem:pointwise-green-estimates-displacement}.
In the cut-off transition region, we have \(|x-x_0|+\rho\eqsim d\).
Thus, the logarithmic factor in the estimate for \(\mathfrak{g}\) is uniformly bounded, and $|D^j\mathfrak{g}(x)|\lesssim d^{-j}$ for $j\geq1$. It then follows from the product rule, the bounds for \(\xi\), and Lemma \ref{lem:pointwise-green-estimates-displacement} that
	\[
	\begin{aligned}
		|D^j\psi (x)|
		\lesssim \sum_{\ell=0}^j|D^\ell\xi(x)||D^{j-\ell}\mathfrak{g} (x)|
		\lesssim \bigl(|x-x_0|+\rho\bigr)^{-j}.
	\end{aligned}
	\]
Finally, in the remaining region where $\xi\equiv0$, we have $\psi\equiv0$, so the asserted estimate follows immediately. This completes the proof.
\end{proof}

\section{A Posteriori Error Estimates for the CR Method}\label{sect:CR} In this section, we prove the localized pointwise a posteriori error estimate for the CR method in Theorem \ref{thm:main_CR}. First, we recall the CR interpolation operator \(I_{h}:H^1_0(\Omega)\longrightarrow V_h\) defined by 
\begin{equation*}
    \int_F I_{h}v \,\dd s = \int_F v \,\dd s, \qquad \forall F \in \mathcal{E}_h .
\end{equation*}
    For every \(T\in\mathcal T_h\) and \(v\in W^{2,1}(T)\), we have the local approximation estimate
	\begin{equation}\label{eq:poisson-interpolation-approx}
		\|v-I_{h}v\|_{L^1(T)}+h_T\| D (v-I_{h}v)\|_{L^1(T)} \lesssim h_T^2\|D^2v\|_{L^1(T)}.
	\end{equation}
Let $W^{k,p}(\mathcal{T}_h)$ denote the piecewise $W^{k,p}$ Sobolev space with respect to $\mathcal{T}_h$. The CR interpolation is similar to a Galerkin projection:
	\begin{equation*}
		a_h\bigl(w_h,v-I_{h}v\bigr)=0,\qquad \forall w_h\in V_h, \qquad \forall v\in H^1_0(\Omega)\cap W^{2,1}(\mathcal T_h).
	\end{equation*}
	The Galerkin projection property and the discrete problem \eqref{eq:CR} yield the identity
	\begin{equation}\label{eq:poisson-residual_id}
		a_h(u-u_h,v)=(f,v-I_{h}v),\qquad \forall v\in H^1_0(\Omega)\cap W^{2,1}(\mathcal T_h).
	\end{equation}

Let $x_0$ be an arbitrary point in $\overline{\omega}$, choose $T_0\in\mathcal T_h$ such that $x_0\in\overline{T}_0$ and $T_0\cap\omega\ne\emptyset$, and recall that $B=B(x_0,d/2)$.
For the CR method, we assume that \(u\in C^{2,\alpha}(T_0)\) for some \(0<\alpha\le1\).
Let \(P(D^{\gamma}u)\) be the Taylor polynomial of \(D^{\gamma}u\) at \(x_0\) of degree \(2-|\gamma|\). 
Set $r=2-|\gamma|$ and $y=x-x_0$. We first derive a pointwise estimate for the difference between $D^{\gamma}u$ and its Taylor polynomial $P(D^{\gamma}u)$ on $T_0$.
	\[
	\begin{aligned}
	D^{\gamma}u(x)-P(D^{\gamma}u)(x)
	&=\frac{1}{(r-1)!}\int_0^1(1-t)^{r-1}\\
	&\quad\times\Bigl(D^{r+\gamma}u(x_0+ty)-D^{r+\gamma}u(x_0)\Bigr)[y^r]\,\mathrm dt\\
    &\lesssim |x-x_0|^{2-|\gamma|+\alpha}|u|_{C^{2,\alpha}(T_0)}, \qquad x\in T_0.
	\end{aligned}
	\]
Since \(\delta_{x_0,\rho}\) reproduces \(P(D^{\gamma}u)\), using the pointwise estimate above, it follows that
\begin{align*}
    \left|D^{\gamma}u(x_0)-\bigl(D^{\gamma}u,\delta_{x_0,\rho}\bigr)_{T_0}\right|
&=\left|\int_{T_0} \bigl(P(D^{\gamma}u)-D^{\gamma}u\bigr)\delta_{x_0,\rho}\,\dd x\right|\\
&\lesssim \rho^{2-|\gamma|+\alpha}|u|_{C^{2,\alpha}(T_0)}.
\end{align*}
Noting that \(\delta_{x_0,\rho}\) also reproduces $D_h^\gamma u_h$ on $T_0$, we obtain
\begin{equation}\label{eq:regularized_evaluation}
	\left|D_h^{\gamma}e(x_0)\right|
	\le \left|\bigl(e,D^{\gamma}\delta_{x_0,\rho}\bigr)_{T_0}\right|
	+C\rho^{2-|\gamma|+\alpha}|u|_{C^{2,\alpha}(T_0)}.
\end{equation}
The second term is the regularization error. \citet{Demlow2007} employed the same technique to derive $W^{1,\infty}$ error estimates for conforming finite element discretizations of Poisson's equation.

\begin{proof}[Proof of Theorem \ref{thm:main_CR}]
Recall that $\phi_m=\xi g_m$ with $m=0$ or 1. Direct calculation shows that
\begin{align*}
    -\Delta &\phi_m = -\xi\Delta g_m - C_{m},\\
    C_{m} &:= 2 D \xi\cdot D  g_m+g_m\Delta\xi.
\end{align*}
Using $-\Delta g_m=D^\gamma\delta_{x_0,\rho}$ with $|\gamma|=m$ and the identity \eqref{eq:poisson-residual_id},
we obtain
\begin{align*}
	(-1)^m(D_h^{\gamma}e,\delta_{x_0,\rho})&=(e,-\Delta g_m)\nonumber\\
	&=(e,-\Delta\phi_m)+(e,C_{m})\nonumber\\
	&=a_h(e,\phi_m)+\sum_{F\in\mathcal E_h}\int_F\llbracket u_h\rrbracket\,\partial_{n}\phi_m\,\mathrm ds+(e,C_{m})\nonumber\\
	&=\bigl(f,\phi_m-I_{h}\phi_m\bigr)+\sum_{F\in\mathcal E_h}\int_F\llbracket u_h\rrbracket\,\partial_{n}\phi_m\,\mathrm ds+(e,C_{m})\\
	&:=I_{m,1}+I_{m,2}+I_{m,3}.
\end{align*}
Since \(\operatorname{supp}\phi_m\subset \omega_{d}\), only elements intersecting \(\omega_{d}\) contribute to \(I_{m,1}\). Using the interpolation error estimate
\eqref{eq:poisson-interpolation-approx}, Lemma \ref{lem:D2phi0}, and Lemma \ref{lem:D2phi-gradient-weighted-L1-estimate}, we obtain
\[
\begin{aligned}
    |I_{m,1}|&\le\sum_{T\cap \omega_{d}\ne\emptyset}\|f\|_{L^\infty(T)}\|\phi_m-I_{h}\phi_m\|_{L^1(T)}\\
	&\lesssim \sum_{T\cap \omega_{d}\ne\emptyset}h_T^2\|f\|_{L^\infty(T)}\|D^2\phi_m\|_{L^1(T)}\\
&\lesssim \max_{T\cap \omega_{d}\ne\emptyset}h_T^{2-m}\|f\|_{L^\infty(T)}\sum_{T\cap \omega_{d}\ne\emptyset}h_T^m\|D^2\phi_m\|_{L^1(T)}\\
&\lesssim L_{\rho,d}\max_{T\cap \omega_{d}\ne\emptyset}h_T^{2-m}\|f\|_{L^\infty(T)}.
\end{aligned}
\]

We define the edge patch and element patch by 
	\begin{align*}
		\Omega_{F}&=\bigcup\{\,K\in\Th:\ F\subset\partial K\,\},\\
	\Omega_{T}&=\bigcup\{\,K\in\Th:\ K = T \ \text{or}\ K\ \text{shares an edge with}\ T\,\}.	
	\end{align*}
For the edge term \(I_{m,2}\), the CR degrees of freedom imply that \(\llbracket u_h\rrbracket\) is orthogonal to the space of constants. Thus, for an arbitrary constant $c_F$,
\[
I_{m,2}=\sum_{F\in\mathcal E_h}\int_F\llbracket u_h\rrbracket\bigl(\partial_{n}\phi_m-c_F\bigr)\,\mathrm ds.
\]
Taking the infimum and using a trace inequality, we have
\begin{align*}
|I_{m,2}|&\le\sum_{F\cap \omega_{d}\ne\emptyset}\|\llbracket u_h\rrbracket\|_{L^\infty(F)}
	\inf_{c_F\in\mathbb R}\|\partial_{n}\phi_m-c_F\|_{L^1(F)}\\
	&\lesssim \sum_{F\cap \omega_{d}\ne\emptyset}\|\llbracket u_h\rrbracket\|_{L^\infty(F)}
	\|D^2\phi_m\|_{L^1(\Omega_{F}\cap B)}.
\end{align*}
For \(m=0\), the finite overlap of the patches \(\Omega_{F}\) and Lemma \ref{lem:D2phi0} yields
\begin{align*}
	|I_{0,2}|&\lesssim \max_{\substack{T\cap \omega_{d}\ne\emptyset\\F\subset\partial T}}
	\|\llbracket u_h\rrbracket\|_{L^\infty(F)}
	\sum_{F\cap \omega_{d}\ne\emptyset}\|D^2\phi_0\|_{L^1(\Omega_{F}\cap B)}\nonumber\\
	&\lesssim L_{\rho,d}\max_{\substack{T\cap \omega_{d}\ne\emptyset\\F\subset\partial T}}
	\|\llbracket u_h\rrbracket\|_{L^\infty(F)}.
\end{align*}
For \(m=1\), an argument analogous to that for \(I_{1,1}\) yields
\begin{equation*}
    |I_{1,2}|\lesssim L_{\rho,d}\max_{\substack{T\cap \omega_{d}\ne\emptyset\\F\subset\partial T}} h_F^{-1}\|\llbracket u_h\rrbracket\|_{L^\infty(F)}.
\end{equation*}

The term \(I_{m,3}=(e,C_m)\) is the pollution term. Let $q$ be the conjugate exponent of $p$. Since $C_m\in C_0^\infty(B)$, the definition of the negative norm gives
\begin{equation*}
|I_{m,3}|\le
\|e\|_{W^{-k,p}(\omega_d)}
\|C_m\|_{W^{k,q}(\omega_d)}.
\end{equation*}
Let $A_\xi:=\{x\in B:c_1d\le |x-x_0|\le c_2d\}$.
Then $\operatorname{supp}C_m\subset A_\xi$ and $|A_\xi|\lesssim d^2$. Moreover, on $A_\xi$, Lemma~\ref{lem:pointwise-green-estimates} and \eqref{eq:cutoff-bounds} imply
\begin{equation*}
|D^j g_m|\lesssim d^{-m-j},\qquad
|D^j\xi|\lesssim d^{-j},
\qquad j\ge0.
\end{equation*}
Consequently, the Leibniz rule yields, for every multi-index $\beta$ with $|\beta|\le k$,
\begin{equation*}
\|D^\beta C_m\|_{L^\infty(A_\xi)}
\lesssim d^{-m-2-|\beta|}.
\end{equation*}
If $1\le q<\infty$, it follows that
\begin{equation*}
\|D^\beta C_m\|_{L^q(\omega_d)}
\le |A_\xi|^{1/q}
\|D^\beta C_m\|_{L^\infty(A_\xi)}
\lesssim d^{-m-2-|\beta|+2/q}.
\end{equation*}
For $q=\infty$, the same estimate holds with $2/q=0$, directly from the preceding pointwise bound. Since $d\le \operatorname{diam}(\Omega)$ and $\Omega$ is fixed, the preceding estimates give
\begin{equation*}
\begin{aligned}
\|C_m\|_{W^{k,q}(\omega_d)}
&\lesssim
\left(\sum_{|\beta|\le k}
d^{q(-m-2-|\beta|+2/q)}\right)^{1/q}
\lesssim d^{-m-k-2+2/q}, &&1\le q<\infty,\\
\|C_m\|_{W^{k,\infty}(\omega_d)}
&\lesssim \max_{|\beta|\le k}d^{-m-2-|\beta|}
\lesssim d^{-m-k-2}, &&q=\infty.
\end{aligned}
\end{equation*}
In either case, $-m-k-2+2/q=-m-k-2/p$. Therefore,
\begin{equation*}
|I_{m,3}|\lesssim
d^{-m-k-2/p}
\|u-u_h\|_{W^{-k,p}(\omega_d)}.
\end{equation*}

Combining the estimates of \(I_{m,1}\)--\(I_{m,3}\) with \eqref{eq:regularized_evaluation} gives
\begin{equation*}
		|D_h^{\gamma}e(x_0)|\lesssim 
		L_{\rho,d}\max_{T\cap \omega_{d}\ne\emptyset}\eta_{|\gamma|}(T)
		+d^{-|\gamma|-k-2/p}\|u-u_h\|_{W^{-k,p}(\omega_{d})}
		+\rho^{\alpha+2-|\gamma|}|u|_{C^{2,\alpha}(\omega_{d})}.
\end{equation*}
Since \(x_0\) is arbitrary and $\rho\lesssim h_{T_0}$, a posteriori estimates for \(\|e\|_{L^\infty(\omega;\mathcal T_h)}\) and \(\|D_h e\|_{L^\infty(\omega;\mathcal T_h)}\) follow.
\end{proof}

It has been shown in \citep{DariDuranPadra2000} that $\eta_0(T)$ is a lower bound for $\|u-u_h\|_{L^\infty(\Omega_T;\mathcal{T}_h)}$ up to data oscillation. Similarly,
using standard bubble-function techniques, we can verify the efficiency of $\eta_1(T)$.
\begin{theorem}\label{thm:poisson-local-efficiency}
		For each \(T\in\mathcal T_h\), let \(f_T\in\mathbb R\) be an arbitrary constant approximation of \(f\) on \(T\). Then, for \(m=0,1\),
		\[
		\eta_m(T)\lesssim \|D_h^m(u-u_h)\|_{L^\infty(\Omega_{T};\mathcal T_h)}+h_T^{2-m}\|f-f_T\|_{L^\infty(T)}.
		\]
	\end{theorem}

\section{A Posteriori Error Estimates for the Morley Method}\label{sect:Morley}
In this section, we prove the localized pointwise a posteriori error estimate for the Morley method in Theorem \ref{thm:main_Morley}. The analysis relies on the Morley interpolation operator \(\Pi_{h}:H^2_0(\Omega)\longrightarrow W_h\) defined by
	\begin{equation*}
		(\Pi_{h}v)(a)=v(a),
		\quad   \forall a \in \mathcal V_h,\qquad
		\int_F \partial_n \Pi_{h}v\,\dd s=\int_F \partial_n v\,\dd s,
		\quad \forall F\in\Eh .
	\end{equation*}
It is known that for each \(T\in\Th\) and \(0\le j\le2\),
	\begin{equation*}
		\|D^j(v-\Pi_{h}v)\|_{L^1(T)}\lesssim h_T^{3-j}\|D^3v\|_{L^1(T)}, \qquad\forall v\in W^{3,1}(T).
	\end{equation*}
The Morley method satisfies an analogous Galerkin projection property:
\begin{equation*}
	b_h\bigl(w_h,v-\Pi_{h}v\bigr)=0,\qquad \forall w_h\in W_h, \qquad \forall v\in H^2_0(\Omega)\cap W^{3,1}(\mathcal{T}_h).
\end{equation*}
This property was proved in \citep{MR3054354}.
Combining it with \eqref{eq:Morley} yields 
\begin{equation}\label{eq:morley-residual_id}
	b_h(U-U_h,v)=(f,v-\Pi_{h}v),\qquad \forall v\in H^2_0(\Omega)\cap W^{3,1}(\mathcal{T}_h).
\end{equation}

\subsection{Weight Functions}
Recall that $h_{T_0}\lesssim\rho\leq \kappa h_{T_0}$, where $\kappa$ is sufficiently small. We introduce the distance functions
\[
a_T:=\dist(x_0,T)+\rho,\qquad a_F:=\dist(x_0,F)+\rho.
\]
Here $a_T$ and $a_F$ are of the same order as $|x-x_0|+\rho$ for $x\in T$ and $x\in F$, respectively. The following novel weight functions are key to our a posteriori error analysis for the Morley method.
\begin{align*}
	\Theta_T:=\frac{h_{T}^2}{a_T(a_T+h_{T})^2},\qquad
	\Phi_T:=\frac12\left(\frac1{a_T^2}-\frac1{(a_T+h_{T})^2}\right).
\end{align*}
The edgewise weights $\Theta_F$ and $\Phi_F$ are defined similarly.
\begin{lemma}\label{lem:sum-weights}
We have
	\begin{equation*}
	\sum_{T\cap B\ne\emptyset}h_T\Theta_T+ \sum_{T\cap B\ne\emptyset}h_T^2\Phi_T \lesssim 1+\log\frac{d}{\rho}.
	\end{equation*}  
\end{lemma}
\begin{proof}
By the definitions of \(\Theta_T\) and \(\Phi_T\), we have
\[
	h_T\Theta_T\lesssim\frac{h_T^3}{a_T^3},\qquad
	h_T^2\Phi_T=\frac{h_T^3(2a_T+h_T)}{2a_T^2(a_T+h_T)^2}\lesssim\frac{h_T^3}{a_T^3}.
\]

For each mesh vertex $z$, define
\begin{equation*}
h_z:=\max\{h_K:K\in\mathcal T_h,\ z\in\overline K\},
\end{equation*}
and let $\widetilde h$ be the continuous piecewise-affine function with nodal values $\widetilde h(z)=h_z$. Shape regularity gives a uniform positive lower bound for the angles of all elements. Conformity allows the elements meeting at a vertex to be ordered so that consecutive elements share an edge, while the angle bound makes the length of this chain uniformly bounded. Since the diameters of two elements sharing an edge are comparable, all element diameters in a vertex patch are uniformly comparable. Consequently,
\begin{equation}\label{eq:lipschitz-meshsize}
\widetilde h|_K\eqsim h_K,
\qquad
\|\nabla\widetilde h\|_{L^\infty(K)}\lesssim1,
\qquad K\in\mathcal T_h.
\end{equation}
The constants depend only on the shape-regularity constant.

Let $T\cap B\ne\emptyset$, and choose $x_T\in\overline T$ such that $|x_T-x_0|=\dist(x_0,T)$. Then $x_T\in B$, and the line segment joining $x_0$ and $x_T$ is contained in $B$. Hence \eqref{eq:lipschitz-meshsize}, $x_0\in\overline T_0$, and $h_{T_0}\lesssim\rho$ imply
\begin{equation}\label{eq:meshsize-distance}
h_T\lesssim \widetilde h(x_T)
\lesssim \widetilde h(x_0)+|x_T-x_0|
\lesssim h_{T_0}+\dist(x_0,T)
\lesssim a_T.
\end{equation}
Thus, by shape regularity,
\begin{equation*}
\frac{h_T^3}{a_T^3}
\lesssim \frac{|T|}{a_T^2}.
\end{equation*}

Choose $M\ge0$ such that
$2^{M-1}\rho<d\le2^M\rho$. Then
$d/2+\rho<2^{M+1}\rho$ and
$M\lesssim1+\log(d/\rho)$. For $i=0,\ldots,M$, set
\begin{equation*}
\mathcal L_i:=\left\{T\in\mathcal T_h:T\cap B\ne\emptyset,\quad
2^i\rho\le a_T<2^{i+1}\rho\right\},
\end{equation*}
so that the sets $\mathcal L_i$ cover all elements intersecting $B$. If $T\in\mathcal L_i$, then \eqref{eq:meshsize-distance} gives
\begin{equation*}
h_T\lesssim a_T\lesssim2^i\rho,
\qquad
T\subset B(x_0,C2^i\rho).
\end{equation*}
Since the element interiors are pairwise disjoint, shape regularity yields
\begin{align*}
\sum_{T\cap B\ne\emptyset}\frac{h_T^3}{a_T^3}
&\lesssim
\sum_{i=0}^M(2^i\rho)^{-2}
\sum_{T\in\mathcal L_i}|T|\\
&\lesssim
\sum_{i=0}^M
\frac{|B(x_0,C2^i\rho)|}{(2^i\rho)^2}
\lesssim1+M
\lesssim1+\log\frac d\rho.
\end{align*}
Combining these estimates completes the proof of Lemma~\ref{lem:sum-weights}.
\end{proof}

\begin{lemma}\label{lem:D3psi_ThetaT}
For each \(T\in\mathcal{T}_h\), we have
	\begin{equation*}
		\|D^3\psi \|_{L^1(T)}
		\lesssim \Theta_T.
	\end{equation*}
\end{lemma}
\begin{proof}
We first consider any far-field triangle $T$ such that \(a_T\geq h_T\). In this case, for every \(x\in T\), using $|x-x_0|+\rho \ge \dist(x_0,T)+\rho = a_T$ and Lemma \ref{lem:Djpsi}, we obtain  
	\[
	\|D^3\psi \|_{L^1(T)} \lesssim \int_T\bigl(|x-x_0|+\rho\bigr)^{-3}\,dx
	\lesssim  h_T^2 a_T ^{-3}\lesssim\Theta_T.
	\]

We then consider any near-field triangle $T$ such that \(a_T < h_T\). 
%In this case, $h_{T}\eqsim h_{T_0}\eqsim\rho$. 
Since \(a_T < h_T\), we have $T\subset B(x_0,2h_T)  \backslash B(x_0,\dist(x_0,T))$. Integrating over the annulus and using Lemma~\ref{lem:Djpsi}, we obtain
	\[
    \begin{aligned}
        \|D^3\psi \|_{L^1(T)}
		&\lesssim \int_{ B(x_0,2h_T)  \backslash B(x_0,\dist(x_0,T))}\frac{1}{\bigl(|x-x_0|+\rho\bigr)^3}\,dx\\%\lesssim \int_{B_{ (C_0+1)h_T }(x_0) \backslash B_{ \dist(x_0,T) }(x_0)}\frac{1}{\bigl(|x-x_0|+\rho\bigr)^3}\,dx
        &\lesssim \int_{\text{dist}(x_0,T)}^{2h_T} \frac{r}{(r+\rho)^3} \dd r
        \lesssim \int_{\text{dist}(x_0,T)}^{2h_T} \frac{1}{(r+\rho)^2} \dd r\\
        &= -\frac{1}{2h_T+\rho}+\frac{1}{a_T}
		\lesssim \frac{1}{a_T }\lesssim \Theta_T.
    \end{aligned}
	\]
The proof is complete.  
\end{proof}

\subsection{Bounds of Edge Jumps}
We need to estimate the terms involving the truncated regularized Green's function on edges.

\begin{figure}[htbp]
    \centering
\begin{tikzpicture}[
scale=0.7, transform shape,
    font=\normalsize,
    triangle/.style={
        draw=black!80,
        fill=gray!10,
        line width=0.8pt
    },
    subtriangle/.style={
        draw=black!65,
        line width=0.55pt
    },
    edge/.style={
        draw=black,
        line width=1.1pt
    },
    projection/.style={
        draw=black!65,
        densely dashed,
        line width=0.65pt
    },
    partition/.style={
        draw=black!75,
        line width=0.55pt
    }
]

%==================================================
% (b) Far field
%==================================================
\begin{scope}

\coordinate (FL) at (-3.5,0);
\coordinate (FR) at ( 3.5,0);
\coordinate (FA) at ( 0,4.8);
\coordinate (Fy) at ( 0.3,0);
\coordinate (Fx) at ( 0.3,3.85);

% Triangle T.
\filldraw[triangle] (FL)--(FR)--(FA)--cycle;

% One subtriangle based on the entire edge F.
\filldraw[subtriangle,fill=bluetwo]
    (FL)--(FR)--(0,1.45)--cycle;

% Distinguished edge F.
\draw[edge] (FL)--(FR);

% Projection of x_0 onto F.
\draw[projection] (Fx)--(Fy);

% Right-angle marker at y_F.
\draw[line width=0.55pt]
    (0.44,0)--(0.44,0.14)--(0.30,0.14);

% Points x_0 and y_F.
\fill[sourcepoint] (Fx) circle (2.2pt);
\fill[black] (Fy) circle (1.5pt);

% Necessary labels.
\node[font=\Large] at (-0.66,2.55) {$T_F$};
\node at (0,-0.55) {$F$};
\node[anchor=west,xshift=-20pt,yshift=2pt] at (Fx) {$x_0$};
\node[anchor=west,font=\large] at (0.48,-0.23) {$y_F$};
\node[font=\Large] at (0.86,0.62) {$K_F$};

\end{scope}
%==================================================
% (a) Near field
%==================================================
\begin{scope}[xshift=8.4cm]

\coordinate (NL) at (-3.5,0);
\coordinate (NR) at ( 3.5,0);
\coordinate (NA) at ( 0,4.8);
\coordinate (Ny) at ( 0,0);
\coordinate (Nx) at ( 0,0.45);

% Triangle T.
\filldraw[triangle] (NL)--(NR)--(NA)--cycle;

% Dyadic subtriangles on the left of y_F.
\filldraw[subtriangle,fill=bluethree]
    (-3.5,0)--(-1.5,0)--(-2.5,0.44)--cycle;
\filldraw[subtriangle,fill=bluetwo]
    (-1.5,0)--(-0.5,0)--(-1.0,0.22)--cycle;
\filldraw[subtriangle,fill=blueone]
    (-0.5,0)--(0,0)--(-0.25,0.11)--cycle;

% Dyadic subtriangles on the right of y_F.
\filldraw[subtriangle,fill=blueone]
    (0,0)--(0.5,0)--(0.25,0.11)--cycle;
\filldraw[subtriangle,fill=bluetwo]
    (0.5,0)--(1.5,0)--(1.0,0.22)--cycle;
\filldraw[subtriangle,fill=bluethree]
    (1.5,0)--(3.5,0)--(2.5,0.44)--cycle;

% Distinguished edge F.
\draw[edge] (NL)--(NR);

% Dyadic partition points.
\foreach \x in {-3.5,-1.5,-0.5,0,0.5,1.5,3.5}
    \draw[partition] (\x,-0.07)--(\x,0.07);

% Projection of x_0 onto F.
\draw[projection] (Nx)--(Ny);

% Right-angle marker at y_F.
\draw[line width=0.55pt]
    (0.14,0)--(0.14,0.14)--(0,0.14);

% Points x_0 and y_F.
\fill[sourcepoint] (Nx) circle (2.2pt);
\fill[black] (Ny) circle (1.5pt);

% Necessary labels.
\node[font=\Large] at (0,2.25) {$T_F$};
\node at (0,-0.58) {$F$};
\node[anchor=west,xshift=-16pt,yshift=2pt] at (Nx) {$x_0$};
\node[anchor=east,font=\large] at (0.2,-0.23) {$y_F$};

% Representative interval labels.
\node[font=\small] at (0.25,-0.24) {$I_0$};
\node[font=\small] at (1.00,-0.24) {$I_1$};
\node[font=\small] at (2.50,-0.24) {$I_2$};

% Representative subtriangle labels.
\node[font=\small] at (0.32,0.25) {$K_0$};
\node[font=\small] at (1.00,0.39) {$K_1$};
\node[font=\small] at (2.50,0.64) {$K_2$};
\end{scope}

\end{tikzpicture}
\caption{Far-field (left) and near-field (right) subtriangle constructions.}
\label{fig:edge-subtriangles}
\end{figure}
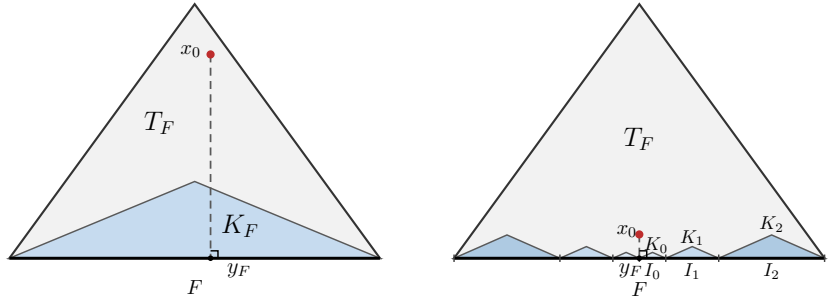

\begin{lemma}\label{lem:partialnn_psi}
	Let \(F \in \mathcal E_h\) be a mesh edge. Then
	\begin{equation*}
		\inf_{c_F\in\mathbb R}
		\|\partial_{nn}\psi -c_F\|_{L^1(F)}
		\lesssim \Theta_F.
	\end{equation*}
\end{lemma}

\begin{proof}
We first consider the far-field edge $F$ such that \(a_F\ge h_F\). We fix a triangle $T_F$ having $F$ as an edge. 
Let \(K_F\) be a subtriangle of $T_F$, as shown in Figure \ref{fig:edge-subtriangles}, such that \(\operatorname{diam}(K_F)\lesssim h_F \), with the height of \(K_F\) chosen sufficiently small, e.g., $h_F/2$.
We have \( |x-x_0|+\rho\gtrsim a_F\) for every \(x\in K_F\).
By a trace inequality and a Poincar\'e inequality,
\[
\inf_{c_F\in\mathbb R} \|\partial_{nn}\psi -c_F\|_{L^1(F)}
\lesssim \|D^3\psi \|_{L^1(K_F)}.
\]
It then follows from the same far-field analysis in Lemma \ref{lem:D3psi_ThetaT} that 
\[
\inf_{c_F\in\mathbb R} \|\partial_{nn}\psi -c_F\|_{L^1(F)}
\lesssim \|D^3\psi \|_{L^1(K_F)}
\lesssim  h_F^2a_F^{-3} \lesssim \Theta_F.
\]
We next consider the near-field edge $F$ such that \(a_F<h_F\). In this case, we set \(c_F=0\). Let $y_F$ be the Euclidean projection of $x_0$ onto $F$.
We will split $F$ into subintervals $I_0, I_1, \ldots, I_J$. 
Each interval $ I_j $ satisfies
    \[
    \begin{aligned}
        |I_j|\eqsim R_j &:=2^j a_F,\qquad  0\le j\le J-1,\\
        |I_J|\le R_J &:=2^J a_F.
    \end{aligned}
    \]

On each subinterval $I_j$, we construct a subtriangle $K_j$ of $T_F$ with height at most $R_j/2$. Figure \ref{fig:edge-subtriangles} shows this construction.
    Since \(R_J\le h_F\), we have \(J\le 1+\log(1+\frac{h_F}{a_F})\).
	For $j\ge 1$ and \(x\in K_j\), since $|y_F-x|\ge \sum_{i=0}^{j-1} R_i\ge R_j/2$, we have 
    \begin{equation}\label{eq:Kj}
        |x-x_0|+\rho\ge  R_j/2\quad\text{in }K_j.
    \end{equation} 
    Since the height of $K_0$ is at most $R_0/2=a_F/2$, the same conclusion also holds for $j=0$.
    Applying the scaled trace inequality on \(K_j\), we obtain
	\[
	\begin{aligned}
		\|\partial_{nn}\psi \|_{L^1(I_j)}
		\lesssim 
		R_j^{-1}\|D^2\psi \|_{L^1(K_j)}
		+\|D^3\psi \|_{L^1(K_j)},\qquad 0\le j\le J.
	\end{aligned}
	\]
Using the pointwise bounds for \(\psi \) in Lemma~\ref{lem:Djpsi} and \eqref{eq:Kj}, we further obtain
	\[
	\begin{aligned}
		\|\partial_{nn}\psi \|_{L^1(I_j)}
		\lesssim \left(R_j^{-1}R_j^{-2} +R_j^{-3} \right)|K_j|
		\lesssim R_j^{-1}.
	\end{aligned}
	\]
Summing the $L^1$ norms over the subintervals yields
	\[
	\|\partial_{nn}\psi \|_{L^1(F)}
	\lesssim \sum_{j=0}^J R_j^{-1}
	\lesssim \sum_{j=0}^J(2^ja_F)^{-1}
	\lesssim a_F^{-1}\lesssim\Theta_F.
	\]
The proof is complete. 
\end{proof}

\begin{lemma}\label{lem:explicit-shear-edge-M}
    Let \(F \in \mathcal E_h\) be a mesh edge. Then
	\begin{equation*}
    \|\partial_{n}\Delta\psi  +\partial_{t} \partial_{nt} \psi \|_{L^1(F)}
	\lesssim \Phi_F.
	\end{equation*}
\end{lemma}
\begin{proof}
Let \(K=K_F\) or $K_j$ be the subtriangle with base \(I=F\) or $I_j$ as constructed in the proof of Lemma \ref{lem:partialnn_psi}. The scaled trace inequality then gives
\begin{equation}\label{eq:shear-scaled-trace}
		\|\partial_{n}\Delta\psi  +\partial_{t} \partial_{nt}\psi \|_{L^1(I)}
		\lesssim |I|^{-1}\|D^3\psi \|_{L^1(K)}
		+\|D^4\psi \|_{L^1(K)}.
	\end{equation}

We first consider the far-field edge $F$ with \(a_F\ge h_F\). Applying \eqref{eq:shear-scaled-trace} with \(K=K_F\) and Lemma~\ref{lem:Djpsi} yields
	\[
    \|\partial_{n}\Delta\psi  +\partial_{t} \partial_{nt}\psi \|_{L^1(F)}
		\lesssim \left( h_F^{-1}a_F^{-3} +a_F^{-4} \right)h_F^2
		\lesssim h_Fa_F^{-3}\lesssim\Phi_F.
	\]

For the near-field case \(a_F<h_F\), let \(F=\bigcup_{j=0}^{J}I_j\) be the dyadic partition introduced in the proof of Lemma~\ref{lem:partialnn_psi}. Applying the scaled trace estimate \eqref{eq:shear-scaled-trace} on each subtriangle \(K_j\) and proceeding as in the far-field case, we obtain
	\[
	\|\partial_{n}\Delta\psi  +\partial_{t} \partial_{nt}\psi \|_{L^1(I_j)}
	\lesssim \left( R_j^{-1}R_j^{-3}+R_j^{-4} \right)R_j^2
	\lesssim R_j^{-2}.
	\]
Summing the $L^1$ norms over the subintervals yields
	\[
	\|\partial_{n}\Delta\psi  +\partial_{t} \partial_{nt}\psi \|_{L^1(F)}
	\lesssim \sum_{j=0}^J \|\partial_{n}\Delta\psi  +\partial_{t} \partial_{nt}\psi \|_{L^1(I_j)}
	\lesssim a_F^{-2}\lesssim\Phi_F.
	\]
This completes the proof.
\end{proof}

\subsection{Proof of the Morley Result} Recall that the regularized Dirac function $\delta_{x_0,\rho}$ is given in \eqref{eq:delta_moments} with $s=3$. Using the same regularization technique as in Section \ref{sect:CR}, we obtain
\begin{equation}\label{eq:regularized_evaluation_Morley}
	\left|D_h^{\gamma}E(x_0)\right|
	\le \left|\bigl(E,D^{\gamma}\delta_{x_0,\rho}\bigr)_{T_0}\right|
	+C\rho^{1+\alpha}|U|_{C^{3,\alpha}(T_0)}.
\end{equation} 
We are now in a position to prove the main theoretical result for the Morley method. 

\begin{proof}[Proof of Theorem \ref{thm:main_Morley}]
For $\psi=\xi\mathfrak{g}$, the product rule implies that
\begin{align*}
    \Delta^2 \psi&=\xi \Delta^2 \mathfrak{g}+\mathcal{C},\\
    \mathcal C &:= \Delta^2\xi \mathfrak{g}+ 2\Delta\xi\,\Delta  \mathfrak{g}+ 4 D (\Delta\xi)\cdot D  \mathfrak{g} + 4 D \xi\cdot D (\Delta \mathfrak{g}) + 4 D^2 \xi :  D^2 \mathfrak{g}. %\nonumber%\label{eq:cutoff-commutator2}
\end{align*} 
Using the weak formulation of the regularized Green's function problem and the identity \eqref{eq:morley-residual_id}, for $|\gamma|=2$ we obtain
\begin{align*}
	(D_h^{\gamma}E,\delta_{x_0,\rho})&=(E,\Delta^2\psi )-(E,\mathcal C)\\
	&=b_h(E,\psi )-\sum_{T\in\mathcal T_h}\int_{\partial T}(D^2\psi \cdot n)\cdot D E\,\mathrm ds\\
    &\qquad+\sum_{T\in\mathcal T_h}\int_{\partial T}E\,\partial_{n}\Delta\psi \,\mathrm ds-(E,\mathcal C)\\
    &=\bigl(f,\psi -\Pi_{h}\psi \bigr)+\Bigl(\sum_{F\in\mathcal E_h}\int_F\llbracket\partial_{n}U_h\rrbracket\,\partial_{nn}\psi \,\dd s\\
    &\qquad-\sum_{F\in\mathcal E_h}\int_F\llbracket U_h\rrbracket \big(\partial_n\Delta\psi +\partial_t\partial_{nt}\psi\big)\,\mathrm ds\Bigr)-(E,\mathcal C)\\
	&=:I_1+I_2+I_3.
	\label{eq:morley-error-split-m}
\end{align*}  
Since \(\operatorname{supp}\psi \subset B\), only elements intersecting \(\omega_{d}\) contribute to \(I_1\). By the interpolation error estimate and Lemmas \ref{lem:D3psi_ThetaT} and \ref{lem:sum-weights}, we have
\begin{align*}
	|I_1|&\le\sum_{T\cap B\ne\emptyset}\|f\|_{L^\infty(T)}\|\psi -\Pi_{h}\psi \|_{L^1(T)}\nonumber\\
	&\lesssim \sum_{T\cap B\ne\emptyset}h_T^3 \|f\|_{L^\infty(T)}\|D^3\psi \|_{L^1(T)}\nonumber\\
	&\lesssim \sum_{T\cap B\ne\emptyset}h_T^2\|f\|_{L^\infty(T)}
	\cdot h_T\Theta_T\\
    &\lesssim L_{\rho,d}\max_{T\cap B\ne\emptyset}h_T^2\|f\|_{L^\infty(T)}.
\end{align*}
As in the CR analysis in the proof of Theorem \ref{thm:main_CR}, the normal-derivative jump of $U_h$ has zero mean on each edge and is orthogonal to edgewise constants.
Therefore, using Lemmas \ref{lem:partialnn_psi}, \ref{lem:explicit-shear-edge-M}, and \ref{lem:sum-weights} yields
\begin{align*}
|I_2|&\le\sum_{F\in\mathcal E_h(B)}\|\llbracket\partial_{n}U_h\rrbracket\|_{L^\infty(F)}\inf_{c_F\in\mathbb R}\|\partial_{nn}\psi -c_F\|_{L^1(F)}\\
	&\quad+\sum_{F\in\mathcal E_h(B)}\|\llbracket U_h\rrbracket\|_{L^\infty(F)}\|\partial_n\Delta\psi +\partial_t\partial_{nt}\psi\|_{L^1(F)}\\
	&\lesssim \sum_{F\in\mathcal E_h(B)}\|\llbracket\partial_{n}U_h\rrbracket\|_{L^\infty(F)}
	\Theta_F+\sum_{F\in\mathcal E_h(B)}\|\llbracket U_h\rrbracket\|_{L^\infty(F)}\Phi_F\\
    &\lesssim L_{\rho,d}\left(
    \max_{F\in\mathcal E_h(B)}h_F^{-1}\|\llbracket\partial_{n}U_h\rrbracket\|_{L^\infty(F)}
    +\max_{F\in\mathcal E_h(B)}h_F^{-2}\|\llbracket U_h\rrbracket\|_{L^\infty(F)}
    \right).
\end{align*}
Here we use the fact that  $\Theta_F\eqsim \Theta_T$ and $\Phi_F\eqsim \Phi_T$ for $F\subset\partial T$.

For the lower-order term $I_3=(E,\mathcal C)$, let $k\ge0$, $1\le p\le\infty$, and let $q$ be the conjugate exponent of $p$. Since $\mathcal C\in C_0^\infty(B)$, the definition of the negative norm gives
\begin{equation*}
|I_3|\le
\|E\|_{W^{-k,p}(\omega_d)}
\|\mathcal C\|_{W^{k,q}(\omega_d)}.
\end{equation*}
Let $A_\xi:=\{x\in B:c_1d\le |x-x_0|\le c_2d\}$.
Then $\operatorname{supp}\mathcal C\subset A_\xi$ and $|A_\xi|\lesssim d^2$. On this transition annulus, Lemma~\ref{lem:pointwise-green-estimates-displacement} and \eqref{eq:cutoff-bounds} give
\begin{equation*}
|D^j\mathfrak g|\lesssim d^{-j},\qquad
|D^j\xi|\lesssim d^{-j},\qquad j\ge0.
\end{equation*}
The Leibniz rule therefore yields
\begin{equation*}
\|D^\beta\mathcal C\|_{L^\infty(A_\xi)}
\lesssim d^{-4-|\beta|},
\qquad |\beta|\le k.
\end{equation*}
For $1\le q<\infty$, multiplication by $|A_\xi|^{1/q}\lesssim d^{2/q}$ gives the corresponding $L^q$ bound; for $q=\infty$, the pointwise bound applies directly. Since $d\le\operatorname{diam}(\Omega)$ and $\Omega$ is fixed, we obtain
\begin{equation*}
\begin{aligned}
\|\mathcal C\|_{W^{k,q}(\omega_d)}
&\lesssim
\left(\sum_{|\beta|\le k}
d^{q(-4-|\beta|+2/q)}\right)^{1/q}
\lesssim d^{-4-k+2/q}, &&1\le q<\infty,\\
\|\mathcal C\|_{W^{k,\infty}(\omega_d)}
&\lesssim \max_{|\beta|\le k}d^{-4-|\beta|}
\lesssim d^{-4-k}, &&q=\infty.
\end{aligned}
\end{equation*}
In either case, $-4-k+2/q=-2-k-2/p$, and hence
\begin{equation*}
|I_3|\lesssim d^{-2-k-2/p}
\|U-U_h\|_{W^{-k,p}(\omega_d)}.
\end{equation*}

Combining the preceding bounds for \(I_1\)--\(I_3\) and \eqref{eq:regularized_evaluation_Morley} yields
	\begin{equation*}
	\begin{aligned}
        |D_h^2E(x_0)|\lesssim  L_{\rho,d} \max_{T\cap \omega_{d}\ne\emptyset} \zeta(T)+d^{-2-k-2/p}\|U-U_h\|_{W^{-k,p}(\omega_{d})}+\rho^{1+\alpha}|U|_{C^{3,\alpha}(\omega_{d})}.
	\end{aligned}
	\end{equation*}
Since \(x_0\) is arbitrary, the a posteriori error estimate in Theorem~\ref{thm:main_Morley} follows. 
\end{proof}

With the help of the bubble functions used in \citep{BeiraoNiiranenStenberg2007}, it is straightforward to show that $\zeta(T)$ provides a lower bound for the local $W^{2,\infty}$ error of the Morley method.
\begin{theorem}\label{thm:morley-local-efficiency}
	For every \(K\in\mathcal T_h\), let \(f_K\in\mathbb R\) be an arbitrary constant approximation of \(f\) on \(K\). Then, for every \(T\in\mathcal T_h\),
    \begin{equation*}
        \zeta(T)\lesssim \|D^2_h(U-U_h)\|_{L^\infty(\Omega_{T};\mathcal T_h)}+\max_{K\subset\Omega_{T}}h_K^{2}\|f-f_K\|_{L^\infty(K)}.
    \end{equation*}
\end{theorem}

\section{Numerical Experiments}\label{sect:NE}
To illustrate the effectiveness of the proposed localized a posteriori error estimates, we present adaptive computations using the CR and Morley methods on the L-shaped domain \(\Omega=(-1,1)^2\setminus ([0,1]\times[-1,0])\).
The target region \(\omega\) and localization radius \(d\) are chosen so that \(\omega_d\) excludes the reentrant corner. Because the exact solutions are unavailable, we evaluate the numerical errors using reference solutions $u_{\rm ref}$ and $U_{\rm ref}$ computed on sufficiently fine meshes. Local mesh refinement is performed by newest-vertex bisection.

\subsection{Control of Pollution Terms}\label{subsect:pollution}
The negative-norm terms $d^{-1-k-2/p}\|u-u_h\|_{W^{-k,p}(\omega_d)}$ and $d^{-2-k-2/p}\|U-U_h\|_{W^{-k,p}(\omega_d)}$ in Theorems~\ref{thm:main_CR} and \ref{thm:main_Morley} are measured in weaker norms and are therefore of higher order. In the computations, we bound these local negative norms by the corresponding global norms $\|u-u_h\|_{W^{-k,p}(\Omega)}$ and $\|U-U_h\|_{W^{-k,p}(\Omega)}$. We take \(d=O(1)\), while retaining flexibility in the choice of the indices $k$ and $p$.

For the CR method, we set $k=0$ and choose either $p=2$ or $p=\infty$ in Theorem \ref{thm:main_CR} and obtain
\begin{equation*}
		\begin{aligned}
				\|D_{h}(u-u_h)\|_{L^\infty(\omega;\mathcal T_h)}
				&\lesssim L_{\underline{h}_\omega,d}\max_{T\cap \omega_{d}\ne\emptyset}\eta_1(T)
				+d^{-1-2/p}\|u-u_h\|_{L^p(\omega_{d})}\\
				&\quad+\max_{T\cap \omega_{d}\ne\emptyset}h_T^{1+\alpha}|  u|_{C^{2,\alpha}(\omega_{d})}.
				\end{aligned}
		\end{equation*}
To control the pollution term $\|u-u_h\|_{L^p(\omega_{d})}$, one may either use an a posteriori estimate for the $L^2$ error $\|u-u_h\|_{L^2(\Omega)}$ or the $L^\infty$ error $\|u-u_h\|_{L^\infty(\Omega)}$.
In the first numerical experiment, we use the following global $L^\infty$ error estimate from \citep{DariDuranPadra2000}:
\[
\|u-u_h\|_{L^\infty(\Omega)}\lesssim\big|\log \max_{T\in\mathcal{T}_h}h_T\big|^{4/3}\max_{T\in\mathcal{T}_h}\eta_0(T).
\]
Alternatively, $\|u-u_h\|_{L^2(\Omega)}$ can be bounded by the $L^2$ error estimator from \citep{CarstensenBartelsJansche2002}:
\[
\tilde{\eta}_0(T)=h_T^2 \|f\|_{L^2(T)} + \sum_{F \in \partial T \cap \mathcal{E}_h^\circ} h_F^{3/2} \|  \llbracket \partial_n u_h\rrbracket\|_{L^2(F)} + \sum_{F \in \partial T} h_F^{3/2} \| \llbracket\partial_t u_h\rrbracket \|_{L^2(F)}.
\]

For the Morley method, we set $k=0$ in Theorem \ref{thm:main_Morley}. Because an a posteriori estimate for $\|U-U_h\|_{L^\infty(\Omega)}$ is not available in the literature, we choose $p=2$ in Theorem \ref{thm:main_Morley} and obtain
\begin{equation*}
\|D_h^2(U-U_h)\|_{L^\infty(\omega;\mathcal T_h)}\lesssim  L_{\underline h_{\omega},d} \max_{T\cap \omega_{d}\ne\emptyset} \zeta(T)+d^{-3}\|U-U_h\|_{L^2(\omega_{d})}+\max_{T\cap \omega_{d}\ne\emptyset} h_T^{1+\alpha}|U|_{C^{3,\alpha}(\omega_{d})}.
\end{equation*}
The pollution term $\|U-U_h\|_{L^2(\omega_{d})}$ is controlled by
\[
\|U-U_h\|_{L^2(\omega_d)}\lesssim\|D_h(U-U_h)\|_{L^2(\Omega)}\lesssim\Big(\sum_{T\in\mathcal{T}_h}\tilde{\zeta}_{1}(T)^2\Big)^\frac{1}{2},
\]
where $\tilde{\zeta}_{1}(T)$ is the broken $H^1$ error estimator
\[
\tilde{\zeta}_1(T) = h_T^3 \|f\|_{L^2(T)} + \sum_{F \in \partial T} h_F^{-1/2} \| \llbracket U_h \rrbracket \|_{L^2(F)} + \sum_{F \in \partial T} h_F^{1/2} \| \llbracket \partial_n U_h \rrbracket \|_{L^2(F)}.
\] 
The broken $H^1$-norm error indicator $\tilde{\zeta}_1(T)$ is obtained from combining the energy-norm estimator in \citep{BeiraoNiiranenStenberg2007} with a duality argument.

\subsection{Poisson Problem}\label{subsec:numerical-poisson}
First, we consider the Poisson equation
\[
	\begin{aligned}
		-\Delta u&=f\quad\text{in }\Omega,\\
		u&=0\quad\text{on }\partial\Omega.
	\end{aligned}
\]
The right-hand side is chosen as
\[
f=10\exp\left(-\frac{(x+0.5)^2+(y-0.3)^2}{0.005}\right)+\cos(\pi x)\cos(\pi y).
\]
The reentrant corner produces a singularity in the solution, and the gradient \(D u\) is unbounded. Therefore, a global \(L^\infty\) estimate for the gradient error is not expected. Instead, we measure the error in the target region \(\omega\).
We choose
\[
\omega=B((-0.25,0.25),0.1),\qquad d=0.1.
\]

Based on the discussion in Section \ref{subsect:pollution}, we use the error indicator
\[
\eta(T)=\left\{\begin{aligned}
    \eta_1(T) + \eta_0(T),\qquad T\cap\omega_d\neq\emptyset,\\
     \eta_0(T),\qquad T\cap\omega_d=\emptyset.
\end{aligned}\right.
\]
In the $\texttt{Solve}\rightarrow\texttt{Estimate}\rightarrow\texttt{Mark}\rightarrow\texttt{Refine}$ adaptive loop, we mark all elements satisfying
\[
\eta(T)\geq 0.6\max_{K\in\mathcal{T}_h}\eta(K).
\]

The log--log plot records $\|D_hu_{\rm ref}-D_hu_h\|_{L^\infty(\omega;\mathcal T_h)}$ and $\|u_{\rm ref}-u_h\|_{L^\infty(\omega;\mathcal T_h)}$, together with the local gradient estimator $\eta_1(\omega_d)=\max_{T\cap\omega_d\neq\emptyset}\eta_1(T)$. Figure \ref{fig:CR} shows that the adaptive method driven by the localized estimator effectively controls the gradient error in $\omega$ by refining the mesh locally in $\omega_d$.
\begin{figure}[htbp]
	\centering
	\begin{minipage}[b]{0.54\textwidth}
		\centering
		\includegraphics[width=\textwidth]{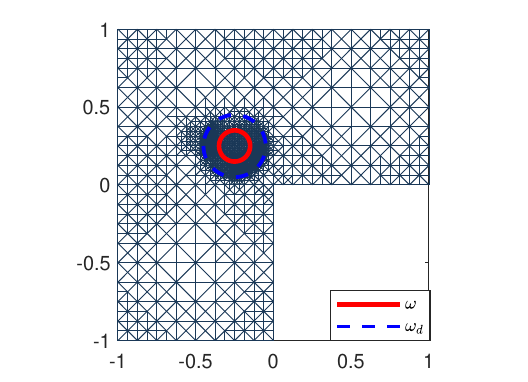}
	\end{minipage}
	\hfill
	\begin{minipage}[b]{0.44\textwidth}
		\centering
		\includegraphics[width=\textwidth]{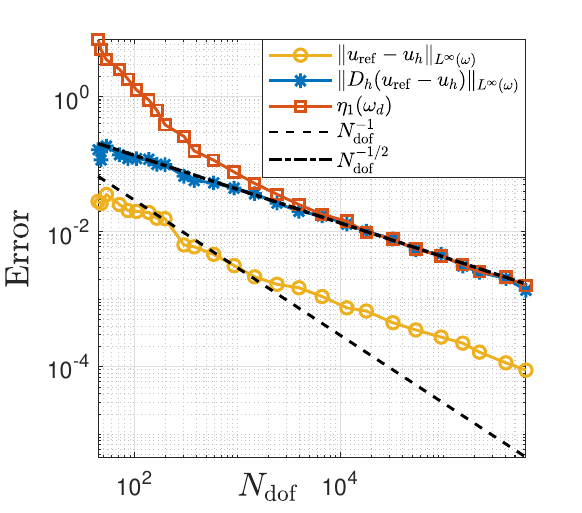}
	\end{minipage}
	\caption{Adaptive mesh (left) and error curves (right) for the CR method.}
	\label{fig:CR}
\end{figure}

\subsection{Biharmonic Problem}\label{subsec:numerical-biharmonic}
Next, we consider the biharmonic equation
	\[
	\begin{aligned}
		\Delta^2 U&=f\quad\text{in }\Omega,\\
		U=\partial_n U&=0\quad\text{on }\partial\Omega.
	\end{aligned}
	\]
The load function \(f\) is the same as in the Poisson case.
The reentrant corner produces a singularity for which the global pointwise norm of \(D^2U\) is unbounded. We therefore focus on the $W^{2,\infty}$ error in the local region
\[
\omega=B((-0.5,-0.2),0.15),\qquad d=0.15.
\]
As discussed in Section \ref{subsect:pollution}, we use $\tilde{\zeta}_1$ to control the pollution error $\|U-U_h\|_{L^2(\Omega)}$. Because $\zeta(T)$ enters through a maximum whereas $\tilde{\zeta}_1(T)$ enters through an $\ell^2$ sum, we use separate marking criteria. We first mark all elements satisfying
\[
\zeta(T)\geq 0.6\max_{K\in\mathcal{T}_h}\zeta(K)\quad\text{ and }\quad T\cap\omega_d\neq\emptyset.
\] 
We then augment the marked set by a minimal set $\mathcal{M}\subset\mathcal{T}_h$ chosen according to the Dörfler criterion:
\[
\sum_{T\in\mathcal M}\tilde{\zeta}_1(T)^2\geq0.6\sum_{T\in\mathcal T_h}\tilde{\zeta}_1(T)^2.
\]

We compare \(\|D_h^2(U_{\rm ref}-U_h)\|_{L^\infty(\omega;\mathcal T_h)}\) and \(\|U_{\rm ref}-U_h\|_{L^\infty(\omega;\mathcal T_h)}\) with the local Hessian estimator in a log--log plot. Figure \ref{fig:Morley} shows that the adaptive method driven by this estimator selectively refines the mesh in $\omega_d$ and effectively reduces the Hessian error in $\omega$.
\begin{figure}[htbp]
	\centering
	\begin{minipage}[b]{0.47\textwidth}
		\centering
		\includegraphics[width=\textwidth]{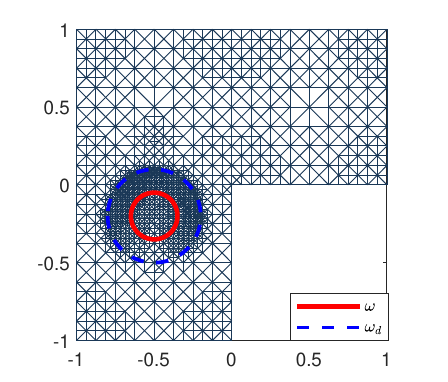}		
	\end{minipage}
	\hfill
	\begin{minipage}[b]{0.44\textwidth}
		\centering
		\includegraphics[width=\textwidth]{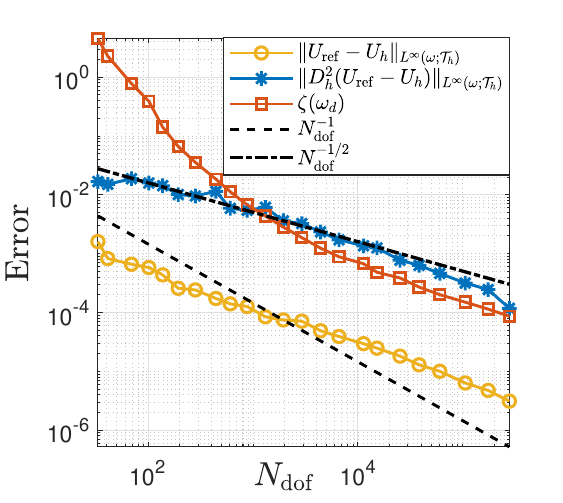}		
	\end{minipage}
	\caption{Adaptive mesh (left) and error curves (right) for the Morley method.}
	\label{fig:Morley}
\end{figure}

\section{Concluding Remarks}\label{sect:conclusion} We have developed localized pointwise a posteriori error estimates for the CR approximation of Poisson's equation and the Morley approximation of the biharmonic equation. Both estimates control nonconforming finite element errors in a target region using local error indicators and higher-order pollution terms. The numerical experiments support the effectiveness of the proposed localized indicators and the corresponding adaptive algorithms.

For the Morley method, the present localized estimator controls the Hessian error. The analysis does not yet provide analogous local or global a posteriori estimates for the function-value error $\|U-U_h\|_{L^\infty}$ or the gradient error $\|D_h(U-U_h)\|_{L^\infty}$. Establishing pointwise a posteriori estimates for these quantities remains a topic for future work.

\section*{Declaration}
During the preparation of this work, the authors used ChatGPT for polishing the writing of the text, generating code, and conducting computational experiments. The authors reviewed and edited the output as needed. The authors assume responsibility for all content.

\section*{Funding}
This work was supported by the National Natural Science Foundation of China under grant 12471346.


\begin{thebibliography}{33}
\providecommand{\natexlab}[1]{#1}
\providecommand{\url}[1]{\texttt{#1}}
\expandafter\ifx\csname urlstyle\endcsname\relax
  \providecommand{\doi}[1]{doi: #1}\else
  \providecommand{\doi}{doi: \begingroup \urlstyle{rm}\Url}\fi

\bibitem[Bank and Li(2019)]{BankLi2019}
Randolph~E. Bank and Yuwen Li.
\newblock Superconvergent recovery of {R}aviart-{T}homas mixed finite elements
  on triangular grids.
\newblock \emph{J. Sci. Comput.}, 81\penalty0 (3):\penalty0 1882--1905, 2019.
\newblock ISSN 0885-7474.
\newblock \doi{10.1007/s10915-019-01068-0}.

\bibitem[Becker et~al.(2010)Becker, Mao, and Shi]{becker2010convergent}
Roland Becker, Shipeng Mao, and Zhongci Shi.
\newblock A convergent nonconforming adaptive finite element method with
  quasi-optimal complexity.
\newblock \emph{SIAM J. Numer. Anal.}, 47\penalty0 (6):\penalty0 4639--4659,
  2010.
\newblock ISSN 0036-1429,1095-7170.
\newblock \doi{10.1137/070701479}.

\bibitem[Beir\~{a}o~da Veiga et~al.(2007)Beir\~{a}o~da Veiga, Niiranen, and
  Stenberg]{BeiraoNiiranenStenberg2007}
L.~Beir\~{a}o~da Veiga, J.~Niiranen, and R.~Stenberg.
\newblock A posteriori error estimates for the {M}orley plate bending element.
\newblock \emph{Numer. Math.}, 106\penalty0 (2):\penalty0 165--179, 2007.
\newblock \doi{10.1007/s00211-007-0066-1}.

\bibitem[Brenner and Scott(2008)]{BrennerScott2008}
Susanne~C. Brenner and L.~Ridgway Scott.
\newblock \emph{The mathematical theory of finite element methods}, volume~15
  of \emph{Texts in Applied Mathematics}.
\newblock Springer, New York, 3 edition, 2008.
\newblock \doi{10.1007/978-0-387-75934-0}.

\bibitem[Cai et~al.(2017)Cai, He, and Zhang]{CaiHeZhang2017}
Zhiqiang Cai, Cuiyu He, and Shun Zhang.
\newblock Residual-based a posteriori error estimate for interface problems:
  nonconforming linear elements.
\newblock \emph{Math. Comp.}, 86\penalty0 (304):\penalty0 617--636, 2017.
\newblock ISSN 0025-5718,1088-6842.
\newblock \doi{10.1090/mcom/3151}.

\bibitem[Carstensen et~al.(2002)Carstensen, Bartels, and
  Jansche]{CarstensenBartelsJansche2002}
Carsten Carstensen, S\"oren Bartels, and Stefan Jansche.
\newblock A posteriori error estimates for nonconforming finite element
  methods.
\newblock \emph{Numer. Math.}, 92\penalty0 (2):\penalty0 233--256, 2002.
\newblock ISSN 0029-599X,0945-3245.

\bibitem[Carstensen et~al.(2013)Carstensen, Gallistl, and Hu]{MR3054354}
Carsten Carstensen, Dietmar Gallistl, and Jun Hu.
\newblock A posteriori error estimates for nonconforming finite element methods
  for fourth-order problems on rectangles.
\newblock \emph{Numer. Math.}, 124\penalty0 (2):\penalty0 309--335, 2013.
\newblock ISSN 0029-599X,0945-3245.
\newblock \doi{10.1007/s00211-012-0513-5}.

\bibitem[Ciarlet(1978)]{Ciarlet1978}
Philippe~G. Ciarlet.
\newblock \emph{The finite element method for elliptic problems}.
\newblock North-Holland Publishing Co., Amsterdam--New York--Oxford, 1978.
\newblock ISBN 0-444-85028-7.
\newblock Studies in Mathematics and its Applications, Vol. 4.

\bibitem[Dall'Acqua and Sweers(2004)]{dall2004estimates}
Anna Dall'Acqua and Guido Sweers.
\newblock Estimates for {G}reen function and {P}oisson kernels of higher-order
  {D}irichlet boundary value problems.
\newblock \emph{J. Differential Equations}, 205\penalty0 (2):\penalty0
  466--487, 2004.
\newblock ISSN 0022-0396,1090-2732.
\newblock \doi{10.1016/j.jde.2004.06.004}.

\bibitem[Dari et~al.(2000)Dari, Dur\'an, and Padra]{DariDuranPadra2000}
E.~Dari, R.~G. Dur\'an, and C.~Padra.
\newblock Maximum norm error estimators for three-dimensional elliptic
  problems.
\newblock \emph{SIAM J. Numer. Anal.}, 37\penalty0 (2):\penalty0 683--700,
  2000.
\newblock ISSN 0036-1429,1095-7170.
\newblock \doi{10.1137/S0036142998340253}.

\bibitem[Demlow et~al.(2012)Demlow, Leykekhman, Schatz, and
  Wahlbin]{DemlowLeykekhmanSchatzWahlbin2012}
A.~Demlow, D.~Leykekhman, A.~H. Schatz, and L.~B. Wahlbin.
\newblock Best approximation property in the ${W}^1_\infty$ norm for finite
  element methods on graded meshes.
\newblock \emph{Math. Comp.}, 81\penalty0 (278):\penalty0 743--764, 2012.

\bibitem[Demlow(2007)]{Demlow2007}
Alan Demlow.
\newblock Local a posteriori estimates for pointwise gradient errors in finite
  element methods for elliptic problems.
\newblock \emph{Math. Comp.}, 76\penalty0 (257):\penalty0 19--42, 2007.
\newblock ISSN 0025-5718,1088-6842.
\newblock \doi{10.1090/S0025-5718-06-01879-5}.

\bibitem[Demlow(2010)]{Demlow2010}
Alan Demlow.
\newblock Convergence of an adaptive finite element method for controlling
  local energy errors.
\newblock \emph{SIAM J. Numer. Anal.}, 48\penalty0 (2):\penalty0 470--497,
  2010.

\bibitem[Demlow and Larsson(2013)]{DemlowLarsson2013}
Alan Demlow and Stig Larsson.
\newblock Local pointwise a posteriori gradient error bounds for the {S}tokes
  equations.
\newblock \emph{Math. Comp.}, 82\penalty0 (282):\penalty0 625--649, 2013.
\newblock ISSN 0025-5718,1088-6842.
\newblock \doi{10.1090/S0025-5718-2012-02647-0}.

\bibitem[Demlow and Makridakis(2010)]{DemlowMakridakis2010}
Alan Demlow and Charalambos Makridakis.
\newblock Sharply local pointwise a posteriori error estimates for parabolic
  problems.
\newblock \emph{Math. Comp.}, 79\penalty0 (271):\penalty0 1233--1262, 2010.
\newblock ISSN 0025-5718,1088-6842.
\newblock \doi{10.1090/S0025-5718-10-02346-X}.

\bibitem[Diening et~al.(2024)Diening, Rolfes, and
  Salgado]{DieningRolfesSalgado2024}
Lars Diening, Julian Rolfes, and Abner~J. Salgado.
\newblock Pointwise gradient estimate of the {R}itz projection.
\newblock \emph{SIAM J. Numer. Anal.}, 62\penalty0 (3):\penalty0 1212--1225,
  2024.
\newblock \doi{10.1137/23M1571800}.

\bibitem[Gastaldi and Nochetto(1987)]{GastaldiNochetto1987}
Lucia Gastaldi and Ricardo~H. Nochetto.
\newblock Optimal ${L}^\infty$-error estimates for nonconforming and mixed
  finite element methods of lowest order.
\newblock \emph{Numer. Math.}, 50:\penalty0 587--611, 1987.
\newblock \doi{10.1007/BF01408578}.

\bibitem[Girault et~al.(2005)Girault, Nochetto, and
  Scott]{GiraultNochettoScott2005}
V.~Girault, R.~H. Nochetto, and L.~R. Scott.
\newblock Maximum-norm stability of the finite element {S}tokes projection.
\newblock \emph{J. Math. Pures Appl.}, 84:\penalty0 279--330, 2005.

\bibitem[Hoffmann et~al.(2001)Hoffmann, Schatz, Wahlbin, and
  Wittum]{HoffmannSchatzWahlbinWittum2001}
W.~Hoffmann, A.~H. Schatz, L.~B. Wahlbin, and G.~Wittum.
\newblock Asymptotically exact a posteriori estimators for the pointwise
  gradient error on each element in irregular meshes. {Part 1}: A smooth
  problem and globally quasi-uniform meshes.
\newblock \emph{Math. Comp.}, 70\penalty0 (235):\penalty0 897--909, 2001.
\newblock \doi{10.1090/S0025-5718-01-01286-8}.

\bibitem[Hu and Ma(2016)]{HuMa2016}
Jun Hu and Rui Ma.
\newblock Superconvergence of both the {C}rouzeix-{R}aviart and {M}orley
  elements.
\newblock \emph{Numer. Math.}, 132\penalty0 (3):\penalty0 491--509, 2016.
\newblock \doi{10.1007/s00211-015-0729-2}.

\bibitem[Hu and Shi(2009)]{HuShi2009}
Jun Hu and Zhongci Shi.
\newblock A new a posteriori error estimate for the {M}orley element.
\newblock \emph{Numer. Math.}, 112\penalty0 (1):\penalty0 25--40, 2009.
\newblock ISSN 0029-599X,0945-3245.
\newblock \doi{10.1007/s00211-008-0205-3}.

\bibitem[Hu et~al.(2021)Hu, Ma, and Ma]{HuMaMa2021}
Jun Hu, Limin Ma, and Rui Ma.
\newblock Optimal superconvergence analysis for the {C}rouzeix-{R}aviart and
  the {M}orley elements.
\newblock \emph{Adv. Comput. Math.}, 47\penalty0 (4):\penalty0 Paper No. 52,
  25, 2021.
\newblock \doi{10.1007/s10444-021-09874-7}.

\bibitem[Krasovskii(1969)]{krasovskii1969properties}
Ju.~P. Krasovskii.
\newblock Properties of {G}reen's functions, and generalized solutions of
  elliptic boundary value problems.
\newblock \emph{Dokl. Akad. Nauk SSSR}, 184:\penalty0 270--273, 1969.
\newblock ISSN 0002-3264.

\bibitem[Leykekhman(2021)]{Leykekhman2021}
Dimitri Leykekhman.
\newblock Pointwise error estimates for ${C}^0$ interior penalty approximation
  of biharmonic problems.
\newblock \emph{Math. Comp.}, 90\penalty0 (327):\penalty0 41--63, 2021.
\newblock \doi{10.1090/mcom/3596}.

\bibitem[Li(2018)]{Li2018SINUM}
Yu-Wen Li.
\newblock Global superconvergence of the lowest-order mixed finite element on
  mildly structured meshes.
\newblock \emph{SIAM J. Numer. Anal.}, 56\penalty0 (2):\penalty0 792--815,
  2018.
\newblock ISSN 0036-1429.
\newblock \doi{10.1137/17M112587X}.

\bibitem[Li et~al.(2007)Li, An, and Li]{LiAnLi2007}
Yuan Li, Rong An, and Kaitai Li.
\newblock Some optimal error estimates of biharmonic problem using conforming
  finite element.
\newblock \emph{Appl. Math. Comput.}, 194:\penalty0 298--308, 2007.
\newblock \doi{10.1016/j.amc.2007.04.040}.

\bibitem[Li(2021)]{Li2021JSCb}
Yuwen Li.
\newblock Recovery-based a posteriori error analysis for plate bending
  problems.
\newblock \emph{J. Sci. Comput.}, 88\penalty0 (3):\penalty0 Paper No. 77, 26,
  2021.
\newblock \doi{10.1007/s10915-021-01595-9}.

\bibitem[Liao and Nochetto(2003)]{LiaoNochetto2003}
X.~Liao and R.~H. Nochetto.
\newblock Local a posteriori error estimates and adaptive control of pollution
  effects.
\newblock \emph{Numer. Methods Partial Differential Equations}, 19\penalty0
  (4):\penalty0 421--442, 2003.
\newblock \doi{10.1002/num.10053}.

\bibitem[Nochetto(1995)]{Nochetto95}
Ricardo~H. Nochetto.
\newblock Pointwise a posteriori error estimates for elliptic problems on
  highly graded meshes.
\newblock \emph{Math. Comp.}, 64\penalty0 (209):\penalty0 1--22, 1995.
\newblock ISSN 0025-5718,1088-6842.
\newblock \doi{10.2307/2153320}.

\bibitem[Rannacher(1979)]{Rannacher1979}
Rolf Rannacher.
\newblock On nonconforming and mixed finite element methods for plate bending
  problems. {The linear case}.
\newblock \emph{RAIRO Mod\'{e}l. Math. Anal. Num\'{e}r.}, 13\penalty0
  (4):\penalty0 369--387, 1979.

\bibitem[Schatz and Wahlbin(1978)]{SchatzWahlbin1978}
Alfred~H. Schatz and Lars~B. Wahlbin.
\newblock Maximum norm estimates in the finite element method on plane
  polygonal domains. {Part 1}.
\newblock \emph{Math. Comp.}, 32\penalty0 (141):\penalty0 73--109, 1978.
\newblock \doi{10.1090/S0025-5718-1978-0502065-1}.

\bibitem[Wang(1993)]{Wang1993}
Ming Wang.
\newblock {$L^\infty$} error estimates of nonconforming finite elements for the
  biharmonic equation.
\newblock \emph{J. Comput. Math.}, 11\penalty0 (3):\penalty0 276--288, 1993.
\newblock ISSN 0254-9409,1991-7139.

\bibitem[Xu and Zhou(2000)]{XuZhou2004}
Jinchao Xu and Aihui Zhou.
\newblock Local and parallel finite element algorithms based on two-grid
  discretizations.
\newblock \emph{Math. Comp.}, 69\penalty0 (231):\penalty0 881--909, 2000.
\newblock \doi{10.1090/S0025-5718-99-01149-7}.

\end{thebibliography}
\end{document}